\documentclass[reqno]{amsart}

\usepackage{amsmath}
\usepackage{amscd}
\usepackage{amssymb}
\usepackage{amstext}
\usepackage{amsmath}
\usepackage[all]{xy}
\usepackage{mathrsfs}
\usepackage[dvipdfmx]{graphicx}

\def\E{\ifmmode{\mathbb E}\else{$\mathbb E$}\fi} %natural numbers
\def\N{\ifmmode{\mathbb N}\else{$\mathbb N$}\fi} %natural numbers%
\def\R{\ifmmode{\mathbb R}\else{$\mathbb R$}\fi} %real numbers
\def\Q{\ifmmode{\mathbb Q}\else{$\mathbb Q$}\fi} %rational numbers
\def\C{\ifmmode{\mathbb C}\else{$\mathbb C$}\fi} %complex numbers
\def\H{\ifmmode{\mathbb H}\else{$\mathbb H$}\fi} %complex numbers
\def\Z{\ifmmode{\mathbb Z}\else{$\mathbb Z$}\fi} %integers
\def\P{\ifmmode{\mathbb P}\else{$\mathbb P$}\fi} %real numbers
\def\T{\ifmmode{\mathbb T}\else{$\mathbb T$}\fi} %real numbers
\def\SS{\ifmmode{\mathbb S}\else{$\mathbb S$}\fi} %real numbers
\def\DD{\ifmmode{\mathbb D}\else{$\mathbb D$}\fi} %real numbers
\def\K{\ifmmode{\mathbb K}\else{$\mathbb K$}\fi}

\theoremstyle{theorem}
\newtheorem{thm}{Theorem}[section]
\newtheorem{cor}[thm]{Corollary}
\newtheorem{lem}[thm]{Lemma}

\newtheorem{prop}[thm]{Proposition}

\newtheorem{conj}[thm]{Conjecture}

\theoremstyle{definition}
\newtheorem{defn}[thm]{Definition}
\newtheorem{rem}[thm]{Remark}

\newtheorem{conds}[thm]{Condition}

\newtheorem{defnlem}[thm]{Definition-Lemma}
\newtheorem{assump}[thm]{Assumption}

\newtheorem{prob}[thm]{Problem}
\newtheorem{proposal}[thm]{Proposal}
\newtheorem{difficulty}[thm]{Difficulty}

\numberwithin{equation}{section}
\begin{document}
\title[categorification of invariants]{Categorification of invariants in 
gauge theory and sypmplectic geometry}
\author{Kenji Fukaya}

\address{Simons Center for Geometry and Physics,
State University of New York, Stony Brook, NY 11794-3636 U.S.A.
\& Center for Geometry and Physics, Institute for Basic Sciences (IBS), Pohang, Korea} \email{kfukaya@scgp.stonybrook.edu}

\begin{abstract}
This is a mixture of survey article and research anouncement.
We discuss Instanton Floer homology for 3 manifolds with boundary. We also discuss
a categorification of the Lagrangian Floer theory using the unobstructed 
immersed Lagrangian correspondence as a morphism in the category of 
symplectic manifolds.
\par
During the year 1998-2012,
those problems have been studied emphasising the ideas from analysis 
such as degeneration and adiabatic  limit (Instanton Floer homology) 
and strip shrinking (Lagrangian correspondence).
Recently we found that replacing those analytic approach 
by a combination of cobordism type argument and 
homological algebra, we can resolve various difficulties 
in the analytic approach.
It thus solves various problems and also simplify many of the 
proofs.
\end{abstract}

\thanks{The author thanks to his joint workers especially 
Aliakebar Daemi, J. Evans, Y. Likili, Y.G.-Oh, H. Ohta and K. Ono for various useful discussions related 
to this topics.}

\maketitle
\tableofcontents

\section{Introduction and review}
\label{section:intro}

The research defining invariants by using moduli spaces in 
differential geometry and topology started around 1980's. 
One of its first example is Donaldson's polynomial invariant of 
smooth 4 manifolds \cite{Don1}.
Various `quantum' invariants of knots which appeared around the same time 
have similar flavor and actually they turn out to be 
closely related to each other.
(Instanton) Floer homology of 3 manifolds (homology 3 spheres) appeared 
late 1980's \cite{fl1} and it was soon realized that Instanton Floer homology provides the basic 
frame work to define a
relative version of Donaldson invariant.
The notion of topological field theory 
was introduced by Witten \cite{witten1} inspired by this relative 
Donaldson invariant.
Soon after that Witten \cite{witten2} found an 
invariant of 3 manifolds (possibly equipped with knot and link) and its relative 
version. This is a generalization of quantum invariant of knot.
The relative version of Witten's invariant uses conformal block as 
its 2 dimensional counterpart.
Segal \cite{seg,seg2} introduced categorical formulation of 
conformal field theory and of several related theories. 
Since then various categorifications 
have been introduced and studied by many 
mathematicians.
In this article the author surveys 
some of them where $A_{\infty}$
category appears.
\par
The gauge theory invariant we discuss 
in this article is one in the column $n=4$,
of the next table.
\par\medskip
\begin{center}
\begin{tabular}{|c|c|c|c|}
\hline
 & invariants & Case $n=4$ & Case $n=3$
\\
\hline
$n$ & number &
Donaldson invariant
& Witten's invariant \\
$n-1$ & group &
Floer homology &  Conformal block\\
$n-2$ & category &
$\frak{Fuk}(R(\Sigma))$ &  representation of loop group \\
\hline
\end{tabular}
\par\medskip
Table 1
\end{center}
\par\medskip
We begin with a quick review of 4 and 3 dimensional invariants.

\subsection{Donaldson invariant}
\label{subsec:Donaldson invariant}

Let $X$ be an oriented 4 manifold
and $\mathcal P_X \to X$  either a principal $SO(3)$
or $SU(2)$ bundle.
(We denote $G =SO(3)$ or $SU(2)$.)
We take a Riemannian metric on $X$, which induces Hodege $*$ 
operator on differential $k$ forms.
$$
* : \Omega^k(X) \to \Omega^{4-k}(X).
$$
On $2$ forms we have $** = 1$.
Therefore $\Omega^2(X) $ is decomposed into a direct 
sum
$$
\Omega^2(X) =\Omega_+^2(X) \oplus \Omega^2_-(X),
\qquad \Omega_{\pm}^2(X) = 
\{  u \in \Omega^2(X) \mid  * u = \pm u\}.
$$
Let $ad(\mathcal P_X) =  \mathcal P_X \times_G \frak g$ be the Lie algebra ($
\frak g = so(3)$ or $su(2)$) bundle 
associated to $\mathcal P_X$ by the adjoint representation $G \to {\rm Aut}(\frak g)$. 
For a connection $A$ of $\mathcal P_X$ its curvature $F_A$ is a section of 
$\Omega^2(X) \otimes ad(\mathcal P_X)$.
We decompose it to 
$$
F_A = F_A^+ + F_A^-
$$
where $F_A^{\pm}$ is a section of $\Omega^2_{\pm}(X) \otimes ad(\mathcal P_X)$.
\par
A connection $A$ is called an Anti-Self-Dual (or ASD) connection if $F_A^+ = 0$.
\par
We denote by $\mathcal A(\mathcal P_X)$ the set of all (smooth) connections 
on $\mathcal P_X$ and $\mathcal G(\mathcal P_X)$ the set of all (smooth)
gauge transformations of $\mathcal P_X$.
(The later is the set of all smooth sections of the  bundle $Ad(\mathcal P_X)
= \mathcal P_X \times_G G$ which is associated 
to $\mathcal P_X$ by the adjoint action of $G$ on $G$.)
The group $\mathcal G(\mathcal P_X)$ acts on $\mathcal A(\mathcal P_X)$ and we denote by $\mathcal B(\mathcal P_X)$
the quotient space.
\par
We denote by $\mathcal M(\mathcal P_X) 
\subset \mathcal B(\mathcal P_X)$ the set of 
all $\mathcal G(\mathcal P_X)$ equivalence classes of ASD connections.
\par
In the simplest case, the Donaldson invariant of $X$ is the order 
of the set $\mathcal M(\mathcal P_X)$
(counted with appropriate sign), and is an integer.
More generally it is regarded as a polynomial map
$
H^*(\mathcal B(\mathcal P_X)) \to \Z
$
obtained by 
\begin{equation}\label{integ11}
c \mapsto \int_{\mathcal M(\mathcal P_X)}c.
\end{equation}
Actually since $\mathcal M(\mathcal P_X)$ is non-compact, 
we need to study the behavior of the cohomology class $c$ 
at infinity of $\mathcal M(\mathcal P_X)$, carefully.
Another problem is that $\mathcal M(\mathcal P_X)$ has 
in general a singularity.
We do not discuss these points in this section.
Donaldson used a map 
$
\mu : H_2(X) \to H^2(\mathcal B(\mathcal P_X)).
$
This map is defined by the slant product $c \mapsto p_1/c$, where 
$p_1$ is 1st Pontriagin or second Chern class of the 
universal bundle on $\mathcal B(\mathcal P_X) \times X$.
\par
On the subring generated by the image of this 
map $\mu$, the integration (\ref{integ11}) behaves nicely and defines an invariant.
(We need to assume that the number $b_2^+$ which is the 
sum of the multiplicities of positive eigenvalues 
of the intersection form on $H_2(X;\Q)$, is not smaller than $2$, 
for this invariant to be well defined.)
In that case, we have a multi-linear map 
on $H_2(X)$ which is called Donaldson's polynomial invariant.
We denote it by
\begin{equation}\label{doninv}
Z(c_1,\dots,c_k;\mathcal P_X) 
=
\int_{\mathcal P_X}\mu(c_1) \wedge \dots \wedge \mu(c_k)
\in \Z
\end{equation}
for $c_i \in H_2(X)$.
Note the integration makes sense only when
$$
\dim \mathcal M(\mathcal P_X) = 2k, 
$$
($\deg c_i = 2$). The dimension of 
$\mathcal M(\mathcal P_X)$ is determined by the seond Chern (or 
the first Pontriagin) number of $\mathcal P_X$.
So if we fix the second Stiefel-Whiteney class of $\mathcal P_X$ the 
isomorphism class of bundle $\mathcal P_X$ 
for which (\ref{doninv}) can be nonzero  is determined by $k$.
So we omit $\mathcal P_X$ and write $Z(c_1,\dots,c_k)$
sometimes.

\subsection{Floer homology (Instanton homology)}
\label{subsec:Floer homology}

Let $M$ be a 3 manifold and $\mathcal P_M$
a principal $G = SO(3)$ or $SU(2)$ bundle on it.
\begin{assump}\label{assumption11}
We assume that one of the following two conditions is satisfied.
\begin{enumerate}
\item
$G = SO(3)$ and $w^2(\mathcal P_M) \ne 0 \in H^2(M;\Z_2)$.
(Here $w^2(\mathcal E_P)$ is the second Stiefel-Whiteney class.)
\item
$G=SU(2)$ and $H(M;\Z) \cong H(S^3;\Z)$.
\end{enumerate}
\end{assump}
The notation $\mathcal A(\mathcal P_M)$, $\mathcal G(\mathcal P_M)$, $\mathcal B(\mathcal P_M)$
are defined in the same way as the 4 dimensional case.
We denote by $R(M;\mathcal P_M) \subset \mathcal B(\mathcal P_M)$ the set of all flat connections.
\par
We assume the following for the simplicity of description.
\begin{assump}\label{assum12}
\begin{enumerate}
\item
The set $R(M;\mathcal P_M)$ is a finite set.
\item
For any $[a] \in R(M;\mathcal P_M)$ the cohomology group
$
H^1(M;ad_a(\mathcal P_M))
$
vanishes.
Here the first cohomology group $H^1(M;ad_a(\mathcal P_M))$ is defined by the 
complex 
$$
ad\, \mathcal P_M \otimes \Omega^0 \overset{d_a }\longrightarrow ad\, \mathcal P_M \otimes \Omega^{1}
 \overset{d_a }\longrightarrow ad\, \mathcal P_M \otimes \Omega^{2}.
$$
(Note $d_a \circ d_a = 0$ since $F_a = 0$.)
\end{enumerate}
\end{assump}
\begin{rem}
We can remove this assumption by appropriately perturbing the defining equation 
$F_a = 0$ of $R(M;\mathcal P_M)$ in a way similar to 
\cite{Don2,fl1,He}.
\end{rem}
In case Assumption \ref{assumption11} (2) is satisfied we put
$R_0(M;\mathcal P_M) = R(M;\mathcal P_M) \setminus \{[0]\}$, 
where $[0]$ is the gauge equivalence class of the trivial connection.
In case of Assumption \ref{assumption11} (1), we put 
$R_0(M;\mathcal P_M) = R(M;\mathcal P_M)$.
\par
We define $\Z_2$ vector space $CF(M;\mathcal P_M)$ whose basis 
is identified with $R(M;\mathcal P_M)$.
We define a boundary operator
$$
\partial : CF(M;\mathcal P_M) \to CF(M;\mathcal P_M)
$$
as follows.
Let $[a], [b] \in R_0(M;\mathcal P_M)$. We fix their representatives $a, b$.
We consider the set of connections $A$ of the 
bundle $\mathcal P_M \times \R$ on $M \times \R$
with the following properties.
We use $\tau$ as the coordinate of  $\R$.
\begin{enumerate}
\item[(IF.1)]
$F_A^+ = 0$.
\item[(IF.2)]
The $L^2$ norm of the curvature 
$$
\int_{M\times \R} \Vert F_A\Vert^2 {\rm vol}_M d\tau
$$
is finite.
\item[(IF.3)]
We require
$$
\lim_{\tau \to - \infty} A = a, \qquad
\lim_{\tau \to + \infty} A = b.
$$
\end{enumerate}
\begin{rem}
We can use Assumption \ref{assum12} (2) to show that 
if  (IF.3) is satisfied then the convergence is automatically 
of exponential order.
\end{rem}
We denote by $\widetilde{\mathcal M}(M\times \R;a,b)$ the set of 
all gauge equivalence classes of the connections $A$
satisfying the above conditions (IF.1), (IF.2), (IF.3).
\par
The $\R$ action induced by the translation of $\R$ factor in $M \times \R$ 
induces an $\R$ action on $\widetilde{\mathcal M}(M\times \R;a,b)$.
We denote the quotient space by this action by ${\mathcal M}(M\times \R;a,b)$.
\begin{thm}\label{thm1-5}{\rm(Floer)}
We can define a map 
$\mu : R_0(M;\mathcal P_M) \to \Z_4$ ($SO(3)$ case) or $\Z_8$ ($SU(2)$ case)
such that:
\begin{enumerate}
\item ${\mathcal M}(M\times \R;a,b)$ is decomposed into pieces
${\mathcal M}(M\times \R;a,b;k)$ where $k+1$ is a natural number congruent to 
$\mu(b) - \mu(a)$.
\item
By `generic' perturbation we may assume that  ${\mathcal M}(M\times \R;a,b;k)$
is compactified to a manifold with corners of dimention $k$,
outside the singularity set of codimension $\ge 2$.
\item
Moreover the boundary of ${\mathcal M}(M\times \R;a,b;k)$ is identified 
with the disjoint union of the direct product
\begin{equation}\label{neweq12}
{\mathcal M}(M\times \R;a,c;k_1) \times {\mathcal M}(M\times \R;c,b;k_2)
\end{equation}
where $c \in R_0(M;\mathcal P_M)$ and $k_1 + k_2 + 1 = k$.
\end{enumerate}
\end{thm}
Now we define 
$$
\langle \partial [a], [b]\rangle
\equiv \# {\mathcal M}(M\times \R;a,c;0) \mod 2,
$$
and 
\begin{equation}
\partial [a]
= \sum_{[b]; \mu(b) = \mu(a) -1} \langle \partial([a]), [b]\rangle[b].
\end{equation}
Theorem \ref{thm1-5} (3) implies that the union of the spaces  (\ref{neweq12})
over $c$ and $k_1,k_2$ with $k_1 + k_2$ given is a boundary of 
some spaces. Especially in case $k_1 = k_2 =0$ the union of 
(\ref{neweq12})
over $c$ is a boundary of $1$ dimensional manifold and so its order is even.
It implies:
$$
\sum_c \# \langle \partial [a], [c]\rangle \langle \partial [c], [b]\rangle = 0.
$$
Namely 
$\partial \partial [a] = 0$.
\begin{defn}\label{defn16}
$$
HF(M;\mathcal P_M) \cong 
\frac{{\rm Ker}\,\,\, \left(\partial : CF(M;\mathcal P_M) \to CF(M;\mathcal P_M\right)}
{{\rm{Im}}\,\,\, \left(\partial : CF(M;\mathcal P_M) \to CF(M;\mathcal P_M)\right)}.
$$
We call this group ($\Z_2$ vector space) the 
{\it Instanton Floer homology} of $(M,\mathcal P_M)$.
Actually we can put orientation of the moduli space we use and 
then define Floer homology group over $\Z$.
\end{defn}
Floer (\cite{fl1, fl2}) proved that the group $HF(M;\mathcal P_M)$ is independent of various 
choices made in the definitions.
\par
The idea behind this definition is to study the following Chern-Simons functional.
We consider the case when $\mathcal P_M$ is an $SU(2)$ bundle
which is necessary trivial as a smooth $SU(2)$ bundle.
We fix a trivialization and then an element of $\mathcal A(\mathcal P_M)$
is identified with an $su(2)$ valued one form $a$.
We may regard it as a $2\times 2$ matrix valued one form.
We then put
\begin{equation}
\frak{cs}(a) = \frac{1}{4 \pi^2}\int_M Tr\left(\frac{1}{2} a\wedge da + \frac{1}{3} a \wedge a\wedge
a\right).
\end{equation}
This functional descents to a map 
$\mathcal B(\mathcal P_M) \to \R/\Z$.
In fact if we regard a gauge transformation as a map $M \to SU(2)$ we have
$$
\frak{cs}(g^*a) = \frak{cs}(a) + \deg g.
$$
On the other hand any connection $A$ of 
$\mathcal P_M \times \R$ on $M \times \R$ can be transformed to a 
connection without $d\tau$ component by a gauge transformation.
(Note $\tau$ is the coordinate of $\R$.)
We call it the {\it temporal gauge}.
If we take the temporal gauge and $A\vert_{M \times \{\tau\}} = a(\tau)$
then the equation $F_A^+ = 0$ is equivalent to
\begin{equation}
\frac{d}{d\tau} a(\tau)  = {\rm grad}_{a(\tau)}\frak{cs}.
\end{equation} 
Here the right hand side is defined by
$$
 \langle{\rm grad}_{a}\frak{cs}, a' \rangle
 =
 \left.\frac{d}{ds} \frak{cs}(a+sa')\right\vert_{s=0}.
$$
($ \langle \cdot, \cdot \rangle$ is the $L^2$ inner product.)
So $HF(M;\mathcal P_M)$ is regarded as a Morse homology 
of $\frak{cs}$.
There is a similar functional in the $SO(3)$ case.

\subsection{Relative Donaldson invariant}
\label{subsec:Relative Donaldson}

The relation between Donaldson invariant and 
Floer homology is described as follows.
\par
Let $X_1$ and  $X_2$ be oriented 4 manifolds with boundary  $M$ and $-M$, 
respectively.
We glue $X_1$ and $X_2$ at $M$ to obtain a closed oriented 4 manifold $X$.
We consider the case $H_1(M;\Q) = 0$. Then
\begin{equation}
H_2(X;\Q) = H_2(X_1;\Q) \oplus H_2(X_2;\Q).
\end{equation}
We also assume $b_2^+(X_i)$, the sum of multiplicities of positive eigenvalues 
of the intersection form on $H_2(X_i;\Q)$, is at least $2$.
\par
We remark that 
$$
R_0(M;\mathcal P_M) \cong R_0(-M;\mathcal P_M).
$$
The map $\R \times M \to \R \times -M$  which sends $(\tau,x)$ to $(-\tau,x)$
is an orientation preserving diffeomorphism. 
So 
$$
\langle \partial_M a, b\rangle = \langle  a, \partial_{-M}b\rangle
$$
Therefore the boundary operator $\partial_{-M}$ is the dual to 
$\partial_{M}$. We thus obtain a pairing:
$$
\langle \cdot,\cdot \rangle : HF(M;\mathcal P_M) \times HF(-M;\mathcal P_M) \to \Z.
$$
\begin{thm}{\rm (Floer-Donaldson, 
See \cite{fu05, Don5})}\label{thm17}
\begin{enumerate}
\item
If $\partial X_1 = M$ and $\left.\mathcal P_{X_1}\right\vert_{M} = \mathcal P_M$, 
then there exists a multilinear map
$$
Z(\cdot;{X_1},\mathcal P_{X_1}) : H_2(X_1;\Z)^{\otimes k} \to HF(M;\mathcal P_M).
$$
\item
In the situation we mentioned at the beginning of this subsection  
we have
\begin{equation}\label{form18}
\aligned
&\langle Z(c_{1,1},\dots,c_{1,k_1};\mathcal P_{X_1}),
Z(c_{2,1},\dots,c_{2,k_2};\mathcal P_{X_2}) \rangle \\
&=
Z(c_{1,1},\dots,c_{1,k_1},c_{2,1},\dots,c_{2,k_2};\mathcal P_{X}).
\endaligned
\end{equation}
\end{enumerate}
\end{thm}
The construction of relative invariant in Theorem \ref{thm17} roughly 
goes as follows.
We take a Riemannian metric on $X_1 \setminus \partial X_1$ such that 
it is isometric to the direct product 
$M \times [0,\infty)$ outside a compact set.
Let $a$ be a flat connection with  $[a] \in R_0(M;\mathcal P_M)$.
\par
We consider the set of connections $A$ of $\mathcal P_X$ such that
\begin{enumerate}
\item
$F_A^+ = 0$.
\item
The $L^2$ norm of the curvature 
$
\int_{X_1} \Vert F_A\Vert^2 {\rm vol}_{X_1}
$
is finite.
\item
We require
$
\lim_{\tau \to + \infty} A\vert_{M\times \{\tau\}} = a.
$
\end{enumerate}
We denote the set of gauge equivalence classes of such $A$ by
$\mathcal M(X_1;a;\mathcal P_{X_1})$ and define
$$
\aligned
&Z(c_{1,1},\dots,c_{1,k_1};{X_1},\mathcal P_{X_1}) \\
&=
\sum_{a  \in R_0(M;\mathcal P_M)} \left(\int_{\mathcal M(X_1;a;\mathcal P_{X_1})}
\mu(c_{1,1}) \wedge \dots \wedge \mu(c_{1,k_1})\right) [a].
\endaligned
$$
We can show that this is a cycle in $CF(M;\mathcal P_M)$ by 
studying the boundary of the moduli space $\mathcal M(X_1;a;\mathcal P_{X_1})$,
that is,
$$
\partial \mathcal M(X_1;a;\mathcal P_{X_1})
= 
\bigcup_b \mathcal M(X_1;b;\mathcal P_{X_1})
\times \mathcal M(M\times \R;b,a).
$$
\par
To show (\ref{form18}) we consider the following sequence of metrics on $X$.
We take compact subsets $K_i$ of $X_i$ such that
$$
{\rm Int}\, X_1 \setminus K_1 \cong M \times (0,\infty),
\qquad {\rm Int}\, X_2 \setminus K_2 \cong M \times (-\infty,0).
$$
We put 
$$
X(T) =  (K_1 \cup  M \times (0,T/2]) \cup (M \times [-T/2,0] \cup K_2)
$$
where we identify $M \times \{T/2\} \cong M \times \{-T/2\}$.
$X(T)$ is diffeomorphic to $X$ and has an obvious Riemannian metric.
So we obtain the moduli space $\mathcal M(X(T);\mathcal P_X)$.
We have
$$
Z(c_{1,1},\dots,c_{1,k_1},c_{2,1},\dots,c_{2,k_2};\mathcal P_{X})
=
\int_{\mathcal M(X(T);\mathcal P_X)} \mu(c_{1,1}) \wedge \dots \wedge \mu(c_{2,k_2})
$$
for any $T$.\footnote{More precisely we cannot expect that 
$\mathcal M(X(T);\mathcal P_X)$ is a smooth manifold for arbitrary $T$.
However we can expect that it is a smooth manifold for $T$ outside a finite set.
But the union of $\mathcal M(X(T);\mathcal P_X)$ for $T \in [T_1,T_2]$ is
again a manifold. So the integral for $T = T_1$ and $T = T_2$ coincides by 
Stokes' theorem.}
\par
Then (\ref{form18}) will be a consequence of the next equality.

\begin{equation}
\lim_{T\to \infty} \mathcal M(X(T);\mathcal P_X)
=
\bigcup_{a \in R_0(M;\mathcal P_M)}
\mathcal M(X_1;a;\mathcal P_{X_1})
\times
\mathcal M(X_2;a;\mathcal P_{X_2}).
\end{equation}

\section{Invariant in dimension 4-3-2}
\label{section:dimension 4-3-2}

The idea to extend the story of subsections 
\ref{subsec:Donaldson invariant}, \ref{subsec:Floer homology}, \ref{subsec:Relative Donaldson}
so that it includes dimension 2 was studied
by various mathematicians in 1990's.
(See for example \cite{fu1,fu2}.)
It can be summarized as follows.
\begin{prob}\label{prob21}
\begin{enumerate}
\item
For each pair of an oriented 2 manifold $\Sigma$ and a 
principal $G$-bundle $\mathcal P_{\Sigma}$ on it, associate a 
category $\mathcal C(\Sigma;\mathcal P_{\Sigma})$,
such that for each two objects of $\mathcal C(\Sigma;\mathcal P_{\Sigma})$ 
the set of morphisms between them is an abelian group.
\item
For any pair $(M,\mathcal P_{M})$ of an oriented $3$ manifold $M$ with boundary 
and a principal $G$-bundle $\mathcal P_{M}$ on it, 
associate an object 
$HF_{(M,\mathcal P_{M})}$ of $\mathcal C(\Sigma;\mathcal P_M\vert_{\Sigma})$,
where $\Sigma = \partial M$.
\item
Let $(M_1,\mathcal P_{M_1})$, $(M_2,\mathcal P_{M_2})$ be pairs as in (2)
such that
$\partial M_1 = - \partial M_2 = \Sigma$,
$\mathcal P_{\Sigma} = \mathcal P_{M_1}\vert_{\Sigma} = 
\mathcal P_{M_2}\vert_{\Sigma}$.
We glue them to obtain $(M,\mathcal P_M)$. Then show:
\begin{equation}
HF(M;\mathcal P_M) = \mathcal C(HF_{(M_1,\mathcal P_{M_1})},
HF_{(M_2,\mathcal P_{M_2})}).
\end{equation}
Here the left hand side is the Instanton Floer homology as in Definition \ref{defn16}
and the right hand side is the set of morphisms in the category 
$\mathcal C(\Sigma;\mathcal P_M\vert_{\Sigma})$, which is an abelian group.
\end{enumerate}
\end{prob}
There is a formulation which include the case
$$
\partial (M,\mathcal P_{M})
=
-(\Sigma_1,\mathcal P_{\Sigma_1})
\sqcup
(\Sigma_2,\mathcal P_{\Sigma_2}).
$$
See Section \ref{DFcategori}.
We may join it with 4+3 dimensional picture 
so that we include the case of 4 manifold with corner.
\par
An idea to find such category $\mathcal C(\Sigma;\mathcal P_{\Sigma})$ 
is based on the fact that the space of all flat connections 
$R(\Sigma;\mathcal P_{\Sigma})$ has a symplectic structure, 
which we define below.
\par
Let $[\alpha] \in R(\Sigma;\mathcal P_{\Sigma})$.
The tangent space
$
T_{\alpha}R(\Sigma;\mathcal P_{\Sigma})
$
is identified with the first cohomology
$$
H^1(\Sigma,ad_{\alpha}(\mathcal P_{\Sigma}))
=
\frac{{\rm Ker} \left(d_{\alpha} : \mathcal P_M \otimes \Omega^{1}
 \overset{d_{\alpha}}\longrightarrow ad\, \mathcal P_M \otimes \Omega^{2}
 \right)}
{{\rm Im} \left(d_{\alpha} : \mathcal P_M  \overset{d_{\alpha}}\longrightarrow ad\, \mathcal P_M \otimes \Omega^{1}
 \right)}.
$$
The symplectic form $\omega$ at $
T_{\alpha}R(\Sigma;\mathcal P_{\Sigma})
\cong H^1(\Sigma,ad_{\alpha}(\mathcal P_{\Sigma}))$ is given by
$$
\omega([u],[v]) = \int_{\Sigma} Tr (u \wedge v).
$$
(See \cite{go}.)
We can prove that this 2 form $\omega$ is a closed two form 
based on the fact that $R(\Sigma;\mathcal P_{\Sigma})$ is regarded as a 
symplectic quotient 
$$
\mathcal A(\Sigma,\mathcal P_{\Sigma})
/\!/
\mathcal G(\Sigma,\mathcal P_{\Sigma}).
$$
In fact we may regard the curvature
$$
\alpha \mapsto F_{\alpha} \in C^{\infty}(\Sigma;ad_{\alpha} \mathcal P \otimes 
\Omega^2)
$$
as the moment map of the action of gauge transformation group 
$\mathcal G(\Sigma,\mathcal P_{\Sigma})$ on 
$\mathcal A(\Sigma,\mathcal P_{\Sigma})$.
(See \cite{AB}.)
\par
We next consider $(M,\mathcal P_{M})$ as in Problem \ref{prob21} (2).
By the same reason as Assumption \ref{assum12} 
we assume:
\begin{assump}\label{assum122}
\begin{enumerate}
\item
The set $R(M;\mathcal P_M)$ has a structure of a finite dimensional manifold.
\item
For any $[a] \in R(M;\mathcal P_M)$ the second cohomology
$
H^2(M;ad_a(\mathcal P_M))
$
vanishes.
Here the second cohomology group $H^2(M;ad_a(\mathcal P_M))$ is the 
cokernel of 
$$
d_{a} : ad\, \mathcal P_M \otimes \Omega^{1}
\longrightarrow ad\, \mathcal P_M \otimes \Omega^{2}.
$$
\end{enumerate}
\end{assump}
\begin{rem}
In Assumption \ref{assum12} we assumed all the cohomology groups 
vanish. In fact in case $\partial M = \emptyset$ 
(and $M$ is 3 dimensional) we have
$$
H^2(M;ad_a(\mathcal P_M))
\cong 
(H^1(M;ad_a(\mathcal P_M)))^*
$$
by Poincar\'e duality. So vanishing of 2nd cohomology
implies the vanishing of the 1st cohomology.
The zero-th cohomology vanishes if the 
connection is irreducible. (Namely the set of all gauge transformations 
which preserve the connection $a$ is zero dimensional.)
\par
We also remark that actually (2) implies (1).
\end{rem}
We then have the next lemma.
\begin{lem}\label{lem2424}
We assume $\partial(M,\mathcal P_M) = (\Sigma,\mathcal P_{\Sigma})$
and Assumption \ref{assum122}.
Let $i : R(M;\mathcal P_M) \to R(\Sigma;\mathcal P_{\Sigma})$
be the map induced by the restriction of the connection.
Then
$$
i^* \omega = 0.
$$
Here $\omega$ is the symplectic form of $R(\Sigma;\mathcal P_{\Sigma})$.
\end{lem}
\begin{proof}
This is an immediate consequence of Stokes' theorem.
\end{proof}
Let $i(a) = \alpha$. We consider the exact sequence
\begin{equation}\label{exact22}
0 \rightarrow H^1(M;ad_a(\mathcal P_M)) \rightarrow 
H^1(\Sigma;ad_{\alpha}(\mathcal P_{\Sigma})) 
\rightarrow
H^2 (M,\Sigma;ad_a(\mathcal P_M))
\rightarrow 
0
\end{equation}
Note
$
H^2(M;ad_a(\mathcal P_M))
\cong 
H^1(M,\Sigma;ad_a(\mathcal P_M)) 
\cong 0
$
by Assumption \ref{assum122} (2) and Poincar\'e duality.
Moreover 
$H^1(M;ad_a(\mathcal P_M)) \cong 
H^2 (M,\Sigma;ad_a(\mathcal P_M))^*$ by Poincar\'e duality.
Thus (\ref{exact22}) implies that
\begin{equation}\label{form23}
\dim R(M;\mathcal P_M) 
= \frac{1}{2} \dim R(\Sigma;\mathcal P_{\Sigma})
\end{equation}
if $\partial M = \Sigma$.
\begin{cor}\label{cor25}
In the situation of Lemma \ref{lem2424}, 
$R(M;\mathcal P_M)$ is an immersed Lagrangian submanifold 
of $R(\Sigma;\mathcal P_{\Sigma})$ if 
$i : R(M;\mathcal P_M) \to R(\Sigma;\mathcal P_{\Sigma})$
is an immersion.
\end{cor}
We can again perturb the defining equation $F_a = 0$ of 
$R(M;\mathcal P_M)$ so that the assumption of Corollary \ref{cor25} 
is satisfied in the modified form.
Namely we obtain a Lagrangian immersion to 
$R(\Sigma;\mathcal P_{\Sigma})$ from a moduli space 
that is a perturbation of $R(M;\mathcal P_M)$.
This is proved by Herald \cite{He}. He also proved that the 
Lagrangian cobordism class of the perturbed immersed 
Lagrangian submanifold is independent of the choice of the perturbation.
\par
The above obserbations let the Donaldson make the next:
\begin{proposal}{(\rm Donaldson \cite{Don})}\label{Dprop}
The category $\mathcal C(\Sigma;\mathcal P_{\Sigma})$
is defined such that:
\begin{enumerate}
\item Its object is a Lagrangian submanifold of 
$R(\Sigma;\mathcal P_{\Sigma})$.
\item
The set of morphisms from $L_1$ to $L_2$ is the 
Lagrangian Floer homology $HF(L_1,L_2)$.
\item
The composition of the morphism is defined by 
counting the pseudo-holomorphic triangle as in 
Figure \ref{zu1} below.
\end{enumerate} 
\par
The first approximation of the object which we assign 
to $(M,\mathcal P_M)$ is the immersed Lagrangian 
submanifold $R(\Sigma;\mathcal P_{\Sigma})$.
\end{proposal}

This proposal is made in 1992 at University of Warwick.
There are various problems to realize this proposal which was known 
already at that stage to experts.
\begin{figure}[h]
\centering
\includegraphics[scale=0.3]{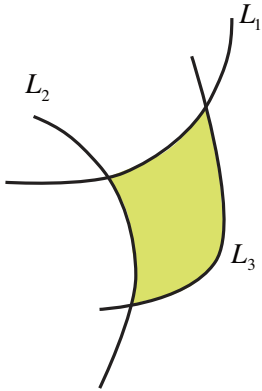}
\caption{Holomorphic triangle}
\label{zu1}
\end{figure}
\begin{difficulty}\label{diff27}
\begin{enumerate}
\item
The space $R(\Sigma;\mathcal P_{\Sigma})$ is 
in general singular. Symplectic Floer theory on 
a singular symplectic manifold is difficult to study.
\item
Even in the case of smooth compact symplectic manifold,
the Floer homology of two Lagrangian submanifolds
is not defined in general and there are various 
conceptional and technical difficulties in 
doing so.
\item
It is known that (Instanton) Floer homology 
of $M = M_1 \cup_{\Sigma} M_2$ is {\it not}
determined by the pair of Lagrangian submanifolds
$R(M_1;\mathcal P_{M_1})$ and 
$R(M_2;\mathcal P_{M_2})$.
So we need some additional information  than 
the immersed Lagrangian submanifold $R(M;\mathcal P_{M})$ to obtain actual relative invariant.
(This is the reason why Donaldson mentioned 
$R(M;\mathcal P_M)$  as a {\it first approximation} and is not 
a relative invariant itself.)
\end{enumerate}
\end{difficulty}
I will explain how much in 25 years, after the  proposal made in 1992, our understanding 
on these problems have been improved.

\section{Relation to Atiyah-Floer conjecture}
\label{section:Atiyah-Floer conjecture}

In this section we explain the relation of the discussion in the previous section 
to a famous conjecture called Atiyah-Floer conjecture \cite{At}.
In its original form Atiyah-Floer conjecture can be stated as follows.
Let $M$ be a closed oriented 3 manifold such that $H_1(M;\Z) = 0$.
We represent $M$ as 
$$
M = H^1_g \cup_{\Sigma_g} H^2_g.
$$
where $H^1_g$ and $H^2_g$ are handle bodies with genus $g$ and 
$\Sigma_g = \partial H^1_g= \partial H^2_g$.
We consider the trivial $SU(2)$ bundle $\mathcal P_M$ on $M$
and on $H^i_g$, $\Sigma_g$. 
Let $R(\Sigma_g;\mathcal P_{\Sigma_g})$, $R(H^i_g;\mathcal P_{H^i_g})$, $R(M;\mathcal P_{M})$ 
be the spaces of gauge equivalence classes of 
flat connections 
of the trivial $SU(2)$ bundle on $H^i_g$, $\Sigma_g$, $M$, respectively.
(\ref{form23}) holds in this case without perturbation.
Namely
\begin{equation}
\dim R(\Sigma_g;\mathcal P_{\Sigma_g}) = 3(g-1) = \frac{1}{2}  \dim  R(H^i_g;\mathcal P_{H^i_g})
\end{equation}
and $R(H^i_g;\mathcal P_{H^i_g})$ $i=1,2$ are Lagrangian `submanifolds' of $R(\Sigma_g;\mathcal P_{\Sigma_g})$.
\begin{conj}{\rm (Atiyah-Floer)}\label{AFconf}
The Instanton Floer homology of $M$ is isomorphic to the Lagrangian Floer homology 
between $R(H^1_g;\mathcal P_{H^1_g})$ and $R(H^2_g;\mathcal P_{H^2_g})$. Namely
\begin{equation}\label{form32}
HF(M;\mathcal P_M) \cong HF(R(H^1_g;\mathcal P_{H^1_g}),R(H^2_g;\mathcal P_{H^2_g})).
\end{equation}
\end{conj}
\begin{rem}
In this remark we mention various problems around Conjecture \ref{AFconf}.
\begin{enumerate}
\item
As we explain in later section (Subsection \ref{subsec:Laggeneral}), 
Lagrangian Floer homology
$HF(L_1,L_2)$  is defined 
as a cohomology of the chain complex whose basis is 
identified with the intersection points 
$L_1 \cap L_2$. (This is the case when $L_1$ is transversal to $L_2$.)
It is easy to see that
$$
R(M;\mathcal P_M) \cong R(H^1_g;\mathcal P_{H^1_g}) \times_{R(\Sigma_g;\mathcal P_{\Sigma_g})} R(H^2_g;\mathcal P_{H^2_g}).
$$
Note instanton Floer homology $HF(M;\mathcal P_M)$ is the homology group 
of a chain complex whose basis is identified with 
$R_0(M;\mathcal P_M)$.
So {\it roughly speaking} two Floer homology groups 
$HF(M;\mathcal P_M)$ and $HF(R(H^1_g;\mathcal P_{H^1_g}),R(H^2_g;\mathcal P_{H^2_g}))$
are homology groups of the chain complexes whose 
underlying groups are isomorphic.
So the important point of the proof of Conjecture \ref{AFconf}
is comparing boundary operators.
\par
The boundary operator defining $HF(M;\mathcal P_M)$ 
is obtained by counting the order of the moduli space ${\mathcal M}(M\times \R;a,b)$, 
as we explained in subsection \ref{subsec:Floer homology}.
The boundary operator defining $HF(R(H^1_g;\mathcal P_{H^1_g}),R(H^2_g;\mathcal P_{H^2_g}))$ is 
obtained by counting the order of the moduli space of 
pseudo-holomorphic strips in $R(\Sigma_g;\mathcal P_{\Sigma_g})$ with boundary condition 
defined by $R(H^1_g;\mathcal P_{H^1_g})$ and $R(H^2_g;\mathcal P_{H^2_g})$.
Various attempts to relate these two moduli spaces directly
have never been successful for more than twenty years.
\item
Another problem, which is actually more serious, is 
related to Difficulty \ref{diff27} (1).
In fact the space $R(\Sigma_g;\mathcal P_{\Sigma_g})$ is singular.
The singularity corresponds to the reducible connections.
(Here a connection $a$ is called reducible if the set of gauge transformations 
$g$ such that $g^*a =a$ has positive dimension.)
Moreover the intersection $R(H^1_g;\mathcal P_{H^1_g}) \cap R(H^2_g;\mathcal P_{H^2_g})$ contains a reducible 
connection, which is nothing but the trivial connection.
Note we assumed Assumption \ref{assumption11} (2).
In this situation the only reducible connection in $R(M;\mathcal P_M)$ is 
the trivial connection. 
The singularity of  $R(\Sigma_g;\mathcal P_{\Sigma_g})$ makes the study of pseudoholomorphic strip in $R(\Sigma_g;\mathcal P_{\Sigma_g})$ with boundary condition 
defined by $R(H^1_g;\mathcal P_{H^1_g})$ and $R(H^2_g;\mathcal P_{H^2_g})$
very hard.
\par
In other words, the right hand side of the isomorphism 
(\ref{form32}) has {\it never} been defined.
In that sense Conjecture \ref{AFconf} is not even a rigorous mathematical 
conjecture yet.
\par
We like to mention that there is an interesting work 
\cite{MW} by Manolescu and Woodward 
on this point.
They used extended moduli space studied previously by 
\cite{Hu,HL,Je}.
A proposal to resolve the problem of singularity 
of $R(\Sigma_g;\mathcal P_{\Sigma_g})$ 
using the idea of \cite{MW} is written in 
\cite{DF}.
%In fact there does not seem to be so much progress on this point 
%in these 25 years.
%On the other hand,
%at this stage of 2016, we have made various 
%progress on item (2)(3) of Difficulty  \ref{diff27}.
%and so it seems very likely that we are in the situation that 
%we can solve Conjecture \ref{AFconf} once we will have a way to `define' 
%the right hand side of (\ref{form32}).
\end{enumerate}
\end{rem}
\par
There are various variants of  Conjecture \ref{AFconf} 
which are solved and/or which can be stated rigorously and/or which are more 
accesible.
\par
Among those variants the most important result is one by Dostglou-Salamon \cite{DS}.
It studies Problem \ref{prob21} 
in the following case.
$M_1 = M_2 = \Sigma \times [0,1]$ where $\Sigma$ is a Riemann surface.
$\mathcal P_{\Sigma}$ is an $SO(3)$ bundle with $w_2(\mathcal P_{\Sigma}) = [\Sigma]$.
Note 
$\partial M_1 = \partial M_2 \cong \Sigma \sqcup -\Sigma$.
When we glue $M_1$ and $M_2$ along their boundaries 
we obtain a closed 3 manifold $M$ of the form
$$
\Sigma \to M \to S^1.
$$
Namely $M$ is a fiber bundle over $S^1$ with fiber $\Sigma$.
The diffeomorphism class of $M$ is determined by 
$\varphi : \Sigma \to \Sigma$. Namely 
$$
M = M_{\varphi} = (\Sigma \times [0,1])/\sim
$$
and the equivalence relation $\sim$ is defined by $(1,x) \sim (0,\varphi(x))$.
\par
The diffeomorphism $\varphi$ induces a diffeomorphism
$$
\varphi^* : R(\Sigma;\mathcal P_{\Sigma}) \to R(\Sigma;\mathcal P_{\Sigma})
$$
which is a symplectic diffeomorphism.
Its graph
$$
{\rm graph}(\varphi^*) 
= \{(x,\varphi^*x) \in R(\Sigma;\mathcal P_{\Sigma})
\times R(\Sigma;\mathcal P_{\Sigma})
\mid x \in R(\Sigma;\mathcal P_{\Sigma})
\}
$$
is a Lagrangian submanifold of $R(\Sigma;\mathcal P_{\Sigma}) \times 
R(\Sigma;\mathcal P_{\Sigma})
$ 
equipped with symplectic form $\omega \oplus -\omega$.
\begin{thm}{\rm(Dostglou-Salamon \cite{DS})}\label{dostsala}
The Instanton Floer homology $HF(M_{\varphi};\mathcal P_{M_{\varphi}})$
is isomorphic to the Lagrangian Floer homology
$HF(\Delta,{\rm graph}(\varphi^*))$,
where $\Delta \subset R(\Sigma;\mathcal P_{\Sigma}) \times 
R(\Sigma;\mathcal P_{\Sigma})$ is the diagonal.
\end{thm}

There is another case which is actually simpler.
We consider $\Sigma = T^2$ (2 torus) with nontrivial $SO(3)$ bundle $\mathcal P_{\Sigma}$.
Then it is easy to see that the space of flat connections 
$R(T^2;\mathcal P_{T^2})$ is one point.
The following is known in this case.

\begin{thm}{\rm (Braam-Donaldson \cite{BD1})}
\label{T2sumthm}
\begin{enumerate}
\item
Let $M$ be an oriented 3 manifolds with boundary such that
each of the connected component of $\partial M$ is $T^2$.
Let $\mathcal P_{M}$ be a principal $SO(3)$ bundle such that 
$w^2(\mathcal P_{M})\vert_{\partial M} = [\partial M]$.
\par
Then we can define a Floer homology $HF(M;\mathcal P_{M})$
which is a $\Z_2$ vector space.
\item
Suppose $M_1$ and $M_2$ are both as in (1).
We assume $\partial M_1 \cong \partial M_2$.
We glue them to obtain $(M,\mathcal P_M)$.
Then
\begin{equation}\label{form333}
HF(M,\mathcal P_M) 
\cong HF(M_1,\mathcal P_{M_1}) \otimes HF(M_2,\mathcal P_{M_2}).
\end{equation}
\end{enumerate}
\end{thm}
Note in the situation of Theorem \ref{T2sumthm} (1) the set of flat connections 
$R(M;\mathcal P_M)$ is a finite set if Assumption \ref{assum122} is satisfied. 
In that case $HF(M;\mathcal P_{M})$ is the cohomology of a chain complex 
$CF(M;\mathcal P_{M})$ whose underlying vector space has a basis identified with $R(M;\mathcal P_M)$.
\par
Note (\ref{form333}) is the case of $\Z_2$ coefficient. 
In the case of $\Z$ coefficient
there is a  K\"unneth type split exact sequence.
\par
In the situation of Theorem \ref{T2sumthm} (2) we assume both $R(M_i;\mathcal P_{M_i})$ ($i=1,2$) satisfy 
Assumption \ref{assum122}. The we can identify
$$
R(M;\mathcal P_M) = R(M_1;\mathcal P_{M_1}) \times R(M_2;\mathcal P_{M_2})
$$
and hence
$$
CF(M;\mathcal P_M) = CF(M_1;\mathcal P_{M_1}) \otimes CF(M_2;\mathcal P_{M_2})
$$
as vector spaces. It is proved in \cite{BD1} that the boundary operators coincide each other.
\par
There are two similar cases which were studied around the same time.
One is the case $\Sigma = S^2$. In this case the bundle $\mathcal P_{S^2}$ on $S^2$ 
is necessary 
trivial if it carries a flat connection.\footnote{If for a pair of closed 3 manifolds $M$ and $SO(3)$ bundle 
$\mathcal P_M$, there exists $S^2 \subset M$ with $w^2(\mathcal P_M) \cap S^2 \ne 0$, then $HF(M;\mathcal P_M) = 0$.}
In this case (\ref{form32}) and (\ref{form333}) correspond to the study of  the 
Floer homology of connected sum. 

\begin{thm}{\rm(Fukaya, Li \cite{fu15,WL})}\label{connectedsum}
Let $(M_1;\mathcal P_{M_1})$ and $(M_2;\mathcal P_{M_2})$  both
satisfy Assumption \ref{assumption11} (2).
We put $M = M_1 \# M_2$ (the connected sum). 
$\mathcal P_{M_1}$ and $\mathcal P_{M_2}$ induce 
a principal bundle $\mathcal P_M$ on $M$ in an obvious way 
so that Assumption \ref{assumption11} (2) is satisfied.
Then for each field $F$ there exists a spectral sequence 
with the following properties:
\begin{enumerate}
\item
Its $E^2$ page is
$$
\aligned
&HF(M_1;\mathcal P_{M_1};F)
\oplus
HF(M_2;\mathcal P_{M_2};F)
\\
&\oplus
\left(
HF(M_1,\mathcal P_{M_1};F) \otimes 
HF(M_2,\mathcal P_{M_2};F) \otimes
H(SO(3);F)
\right).
\endaligned
$$
\item
It converges to $HF(M;\mathcal P_{M};F)$.
\end{enumerate}
\end{thm}
We can prove a similar statement in the case of  Assumption \ref{assumption11} (1).
(It was not explored 25 years ago.)
\par
Note in the situation of Theorem \ref{connectedsum}
we have the following isomorphism if Assumption \ref{assum122} is satisfied. 
\begin{equation}\label{form34}
\aligned
R_0(M;\mathcal P_{M})
\cong
&\,\,R_0(M_1;\mathcal P_{M_1})
\sqcup
R_0(M_2;\mathcal P_{M_2}) \\
&\sqcup
(R_0(M_1;\mathcal P_{M_1})
\times
R_0(M_2;\mathcal P_{M_2})
\times SO(3)).
\endaligned
\end{equation}
Note $R_0(M;\mathcal P_{M}) = R(M;\mathcal P_{M}) \setminus \{\text{trivial connection}\}$
in our situation. 
The first and the second term of the right hand side of 
(\ref{form34}) correspond to the flat connection on $M$ which is 
trivial either on $M_1$ or on $M_2$.
The third term of (\ref{form34})  corresponds to the flat connection on $M$ which is 
nontrivial both on $M_1$ and $M_2$. In this case there is extra freedom 
to twist the connections on $S^2$ where we glue $M_1$ and $M_2$.
It is parametrized by $SO(3)$.
(\ref{form34}) explains Theorem \ref{connectedsum} (1).
\par
The spectral sequence in general does not degenerate in $E^2$ level.
In fact, there is a nontrivial differential which is 
related to the fundamental homology class of $H(SO(3))$.
One such example is the case when $M_1$ is Poincar\'e homology sphere and
$M_2$ is Poincar\'e homology sphere with reverse orientation.
\par
The next simplest case is one when $\Sigma = T^2$ with the trivial 
$SU(2)$ bundle.
The following Floer's exact sequence is 
closely related to this case.
Suppose $(M;\mathcal P_{M})$ is as in Assumption \ref{assumption11} (1)
and we take $S^1 = K \subset M$, a knot.
We remove a tubular neighborhood $S^1 \times D^2$ of $K$ from 
$M$ and re-glue $S^1 \times D^2$ along the boundaries to obtain $M'$.
This process is called Dehn surgery.
There are several different ways to re-glue, 
which is parametrized by the diffeomorphism $T^2 \to T^2$.
Composing with the diffeomorphism which extends to $S^1 \times D^2$ does not 
change the diffeomorphism type of $M'$ and so $M'$ is parametrized by a pair of integers
$(p,q)$
(which are coprime) or $p/q \in \Q \cup \{\infty\}$.
We consider the case when this rational number is $-1,1,0$ and write them $M_{-1}$,
$M_0$ and $M_1$.  The manifold $M_{-1}$ is actually $M$ itself.
$M_{+1}$ is another 3 manifold which is a homology 3 sphere. (It satisfies 
Assumption \ref{assumption11} (1).)
$M_0$ is homology $S^1 \times S^2$.
We can extend $\mathcal P_{M}\vert_{M \setminus K}$ to it 
and obtain $\mathcal P_{M_0}$ such that the flat connection 
of $\mathcal P_{M_0}$ corresponds to the 
group homomorphism $\pi_1(M \setminus K) \to SU(2)$ 
which sends the meridian to $-1$. Here meridian is a small
circle which has liking number 1 with the knot $K$.
\begin{thm}{\rm(Floer \cite{fl2}, Braam-Donaldson \cite{BD1})}
\label{Ftriangle}
There exists a long exact sequence:
\begin{equation}\label{triangleFL}
\longrightarrow HF(M_{-1};\mathcal P_{M_{-1}})
\longrightarrow HF(M_{0};\mathcal P_{M_{0}})
\longrightarrow HF(M_{+1};\mathcal P_{M_{+1}})
\longrightarrow
\end{equation}
\end{thm}
The relation of Theorem \ref{Ftriangle} to the gluing problem such as 
Theorems \ref{dostsala}, \ref{T2sumthm}, \ref{connectedsum} is as follows.
We put
$$
\overset{\circ} M = M \setminus (S^1 \times {\rm Int} \, D^2).
$$
(Here $S^1 \times D^2$ is the tubular neighborhood of the knot $K$.)
The boundary $\partial \overset{\circ} M$ is $T^2$ on which $\mathcal P_M$ is trivial.
Therefore $R(\partial \overset{\circ} M,\mathcal P_{\partial \overset{\circ} M})$
is the set of gauge equivalence classes of flat $SU(2)$ connections on $T^2$, 
which is identified with $T^2/{\pm 1}$.
So, according to Proposal \ref{Dprop}, the relative invariant 
$HF_{\overset{\circ} M}$ `is' a Lagrangian submanifold of $T^2/{\pm 1}$, 
which is a sum of immersed circles in it.
\par
On the other hand, the manifolds $M_{-1}, M_{0}, M_{+1}$ are obtained by glueing $S^1 \times D^2$ in various ways to $\overset{\circ} M$.
The relative invariant $HF_{S^1 \times D^2}$  is the set of flat connections 
on $M$  which can be identified with various circles.
Let $C_{-1}, C_0$ and $C_{+1}$ be the circules in $T^2/{\pm 1}$ corresponding 
to $-1$, $0$ and $+1$ surgeries, respectively.
\par
The Floer homologies appearing in (\ref{triangleFL}) is obtained 
as the set of `morphisms' from the object $HF_{\overset{\circ} M}$ to those circles
$C_{-1}, C_0$ and $C_{+1}$.
The proof then goes by using the identity
$$
[C_{-1}] + [C_{+1}] = [C_{0}]
$$
as cycles.
\begin{figure}[h]
\centering
\includegraphics[scale=0.3]{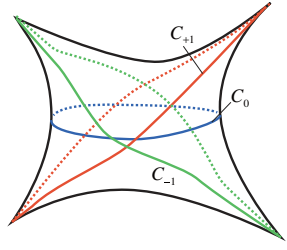}
\caption{$C_{\pm 1}, C_0$}
\label{zu2}
\end{figure}

\section{Biased review of Lagrangian Floer theory I}
\label{section:Lagrangian Floer theory}

In this and the next sections we review Lagrangian Floer theory with emphasis on its application to
gauge theory.
Another review of Lagrangian Floer theory which put more emphasis on its application 
to Mirror symmetry is \cite{fu055}.

\subsection{General idea of Lagrangian Floer homology}
\label{subsec:Laggeneral}

Let $(X,\omega)$ be a $2n$ dimensional compact symplectic manifold (namely $X$ is a 
$2n$ dimensional manifold 
and $\omega$ is  a closed 2 form on it such that $\omega^n$ never vanish.)
Let $L_1, L_2$ be Lagrangian submanifolds of $X$ (namely they are 
$n$ dimensional submanifolds such that $\omega\vert_{L_i} = 0$.)
We first consider the case when $L_1$ and $L_2$ are both embedded.
We assume for simplicity that $L_1$ is transversal to $L_2$.
It implies that the set $L_1 \cap L_1$ is a finite set.
Let 
$CF(L_1,L_2;\Z_2)$ be the $\Z_2$ vector space
whose basis is identified with $L_1 \cap L_2$.
We take and fix an almost complex structure $J$ which is 
compatible with $\omega$. (Namely we assume $\omega(JX,JY) = \omega(X,Y)$ 
and $g(X,Y) = \omega(X,JY)$ is a Riemannian metric.\footnote{Actually the first 
condition is a consequence of the second condition.})
For $a,b \in L_1 \cap L_2$ we consider the set of maps
$$
u : \R \times [0,1] \to X
$$
with the following properties. (Here $\tau$ and $t$ are 
coordinate of $\R$ and $[0,1]$ respectively.)
\begin{enumerate}
\item[(LF.1)]
$u$ is $J$ holomorphic. Namely
\begin{equation}\label{holoeq}
\frac{\partial u}{\partial\tau} = J\frac{\partial u}{\partial t}.
\end{equation}
\item[(LF.2)]
$u(\tau,0) \in L_0$ and $u(\tau,1) \in L_1$.
\item[(LF.3)]
$$
\int u^* \omega < \infty.
$$
Moreover
$$
\lim_{\tau\to -\infty}u(\tau,t) = a,
\qquad
\lim_{\tau\to +\infty}u(\tau,t) = b.
$$
\end{enumerate}
We remark that these conditions are similar to 
(IF.1), (IF.2), (IF.3) we put to define the moduli space ${\mathcal M}(M\times \R;a,b)$
in Subsection \ref{assumption11}.
We denote the set of the maps $u$ satisfying these conditions 
by $\widetilde{\mathcal M}(L_1,L_2;a,b)$.
The translation on $\R$ direction of the source $\R \times [0,1]$ 
induces an $\R$ action on $\widetilde{\mathcal M}(L_1,L_2;a,b)$.
We denote by ${\mathcal M}(L_1,L_2;a,b)$ the quotient space of this action.
\par
We decompose ${\mathcal M}(L_1,L_2;a,b)$ as
\begin{equation}
{\mathcal M}(L_1,L_2;a,b)
=
\bigcup_{k,E} {\mathcal M}(L_1,L_2;a,b;k,E)
\end{equation}
where $k \in \Z$ and $E \in \R_{\ge 0}$. 
Here ${\mathcal M}(L_1,L_2;a,b;k,E)$ consists of the maps $u$ such that:
\begin{enumerate}
\item
$\int u^*\omega = E$.
\item
The index of the linearized equation (\ref{holoeq}) at $u$ is $k$.
\end{enumerate}
Roughly speaking the Floer's boundary operator is defined by
\begin{equation}\label{form43}
\langle \partial a,b \rangle
=
 \sum_{E}\#{\mathcal M}(L_1,L_2;a,b;1,E) [b].
\end{equation}
The proof of $\partial\partial = 0$ would be based on the equality
\begin{equation}\label{neweq122}
\aligned
&\partial {\mathcal M}(L_1,L_2;a,c;2,E) 
\\
&=
\bigcup_{E_1 + E_2 = E}{\mathcal M}(L_1,L_2;a,c;1,E_1) 
\times {\mathcal M}(L_1,L_2;c,b;1,E_2),
\endaligned
\end{equation}
which is similar to (\ref{neweq12}).
Actually (\ref{neweq122}) does not hold in general.
There is so called disk bubble 
which corresponds to another type of the boundary component of 
${\mathcal M}(L_1,L_2;a,c;2,E)$.
Floer \cite{Flo88IV} put Condition \ref{conds41} below to avoid it.
\begin{conds}\label{conds41}
For any $u : (D^2,\partial D^2) \to (X,L_i)$ ($i=1,2$) we have
$$
\int_{D^2} u^* \omega = 0.
$$
\end{conds}
\begin{thm}{\rm (Floer \cite{Flo88IV})}\label{FloerLag}
Under Condition \ref{conds41}, the following holds for generic compatible 
almost complex structure $J$.
\begin{enumerate}
\item The moduli space ${\mathcal M}(L_1,L_2;a,b;1,E)$ 
satisfies an appropriate transversality condition.
\item
We can define the boundary operator by (\ref{form43}).
\item
(\ref{neweq122}) holds and we can use it to prove $\partial\partial = 0$.
So we can define Floer homology
$$
HF(L_1,L_2) \cong 
\frac{{\rm Ker}\,\,\, (\partial : CF(L_1,L_2) \to CF(L_1,L_2))}
{{\rm{Im}}\,\,\, (\partial : CF(L_1,L_2) \to CF(L_1,L_2))}.
$$
\item
If $\varphi : X \to X$ is a Hamiltonian diffeomorphism then
$$
HF(\varphi(L_1),L_2) \cong HF(L_1,L_2).
$$
\item
If $L_1 = L_2 = L$ and $\varphi : X \to X$ is a Hamiltonian diffeomorphism
then
$HF(\varphi(L),L)$ is isomorphic to the ordinary homology $H(L;\Z_2)$ of $L$.
\end{enumerate}
\end{thm}
\begin{rem}
We do not explain the notion of Hamiltonian diffeomorphism
here since our description of Lagrangian Floer homology 
is biased to the direction which is related to Gauge theory.
\par
If $L_i$ is spin we can define Floer homology over $\Z$ coefficient.
(We can relax this condition to relative spinness.) See \cite[Chapter 8]{fooobook2}. 
\end{rem}

\subsection{Monotone Lagrangian submanifold}
\label{subsec:monotone}
Condition \ref{conds41} is too much restrictive.
Especially  we can not work under this condition in our application
to gauge theory (for example to realize Proposal \ref{Dprop}).
The next step to relax this condition is due to Y.-G.Oh.
\begin{thm}{\rm(Oh \cite{Oh})}\label{monotonecaseoh}
Instead of Condition \ref{conds41} we assume that 
$L_i$ are monotone and has minimal Maslov number $>2$, for $i=1,2$.
Then Theorem \ref{FloerLag} (1)(2)(3)(4) holds.
\end{thm}
We will explain the notion of monotonicity and minimal Maslov
number later in this subsection.
\begin{rem}
\begin{enumerate}
\item
Under the assumption of Theorem \ref{monotonecaseoh},
Theorem \ref{FloerLag} (5) may not hold.
There exists however a spectral sequence whose $E^2$ term is 
$H(L;\Z_2)$ and which converges to $HF(\varphi(L),L)$.
\item
If we assume $L_i$ to be spin (or more generally $(X,L_1,L_2)$ are relatively 
spin) Theorem \ref{monotonecaseoh} (and item (1) of this remark) 
holds over $\Z$ coefficient.
This is proved in \cite[Chapters 2 and 6]{fooobook} and \cite[Chapter 8]{fooobook2}.
\end{enumerate}
\end{rem}
To apply Theorem \ref{monotonecaseoh} to gauge theory, we can use the next fact.
\begin{prop}\label{prop46}
Let $(M,\mathcal P_M)$ be a pair of an oriented 3 manifold with boundary 
and a principal $SO(3)$ bundle on it such that
$w^2(\mathcal P_M)\vert_{\partial M}$ is the fundamental class 
$[\partial M]$.
We also assume Assumption \ref{assum122}.
Moreover we assume $R(M;\mathcal P_M) \to R(\partial M;\mathcal P_{\partial M})$
is an {\it embedding} (that is, injective).
\par
Then $R(M;\mathcal P_M)$ is a monotone Lagrangian submanifold.
Its minimal Maslov number is congruent to $0$ modulo $4$.
\end{prop}
This fact is known to many researchers at least in late 1990's.
\par
\begin{thm}\label{embedcase}
Suppose $(M_i,\mathcal P_{M_i})$ satisfies the assumptions of 
Proposition \ref{prop46} for $i=1,2$. 
Theorem \ref{monotonecaseoh} implies that Floer homology
$HF(R(M_1;\mathcal P_{M_1}),R(M_2;\mathcal P_{M_2}))$ is well-defined.
We assume that $\partial M_1 = - \partial M_2$ and 
$\mathcal P_{M_1}\vert_{\partial M_1} \cong \mathcal P_{M_2}\vert_{\partial M_2}$.
\par
Then Instanton Floer homology is isomorphic to Lagrangian Floer homology:
\begin{equation}\label{form4545}
HF(M;\mathcal P_M) \cong HF(R(M_1;\mathcal P_{M_1}),R(M_2;\mathcal P_{M_2})).
\end{equation}
Here $(M,\mathcal P_M)$ is obtained by gluing 
$(M_1,\mathcal P_{M_1})$ and $(M_2,\mathcal P_{M_2})$ along their boundaries.
\end{thm}
\begin{rem}
Theorem \ref{embedcase} is claimed as \cite[Corollary 1.2]{fu7}.
The outline of its proof is given in \cite[Section 5]{fu7}.
The detail of the proof of Theorem \ref{embedcase} will be 
written in a subsequent joint paper \cite{DFL} with Aliakebar Daemi 
and Max Lipyanskiy.
\par
The cobordism argument used in \cite[Section 5]{fu7} 
appeared also in \cite[Section 8]{fu2} as the proof of
\cite[Theorem 8.7]{fu2}, which claims that there exists a 
homomorphism from the left hand side of (\ref{form4545}) 
to the right hand side of (\ref{form4545}).
It was conjectured but not proved in 
\cite[Conjecture 8.9]{fu2} that this homomorphism 
is an isomorphism.
We use the same map to prove Theorem \ref{embedcase}.
The idea which was missing in 1997 when 
\cite{fu2} was written is the following.
\begin{enumerate}
\item
When $\partial(M,\mathcal P_{M}) = (\Sigma,\mathcal P_{\Sigma})$,
we 
use $R(M;\mathcal P_{M})$ itself as a `test object'
of the functor: $\frak{Fuk}(R(\Sigma:\mathcal P_{\Sigma}))
\to \mathscr{CH}$. 
(Here $\mathscr{CH}$ is the $A_{\infty}$ category 
whose object is a chain complex.)
This is the functor 
which associates to $L$ (a Lagrangian submanifold 
of $R(\Sigma:\mathcal P_{\Sigma})$) the chain complex 
$CF((M;\mathcal P_{M});L)$ the  
`Floer's chain complex of $(M;\mathcal P_{M})$ with boundary 
condition given by $L$'.
(See Theorem \ref{thm612}.)
In other words we take $R(M;\mathcal P_{M})$ as 
the Lagrangian submanifold $L$ in (\ref{eq6666}).
\item
If we take $L = R(M_i;\mathcal P_{M_i})$ then 
the chain complex $CF((M;\mathcal P_M);L)$ 
in (\ref{eq6666}) can be identified with 
the De-Rham complex of $L$.
\end{enumerate}
Something equivalent to these two points 
appeared in the paper \cite{LL} by Lekili-Lipyanskiy.
After that it was used more explicitly in \cite{fu7}.
The same argument applied  
in the a similar situation as Theorem \ref{embedcase} was
explicitly mentioned in a talk by Lipyanskiy \cite{Li2}
done in 2012.
\par
It seems to the author that \cite{LL} is the paper which
revives the idea using the cobordism argument
in this and related problems and became the
turning point of the direction of the research.
During the years 1998-2010 
the cobordism argument proposed in 
\cite{fu2} was not studied and instead 
more analytic approach using adiabatic 
limit had been persued.
\end{rem}
In the rest of this subsection we explain the notion of 
montone Lagrangian submanifold and minimal Maslov number 
and a part of the idea of the proof of Theorem \ref{monotonecaseoh}.
\par
We first review the definition of Maslov index of a Lagrangian submanifold.
(See \cite[Subsection 2.1.1]{fooobook} for detail.)
Let $(X,L)$ be a pair of a symplectic manifold $X$ and its 
(embedded) Lagrangian submanifold $L$.
We take a compatible almost complex structure $J$ on $X$.
For a map $u : (D^2,\partial D^2) \to (X,L)$ we consider 
the equation $\overline\partial u =0$.
Its linearlization defines an operator
$$
D_u\overline{\partial} :
C^{\infty}((D^2,\partial D^2);(u^*TX,u^*TL))
\to C^{\infty}(D^2;u^*TX \otimes \Omega^{0,1}).
$$
We can show that there exists $\mu : H_2(X,L) \to \Z$ such that
\begin{equation}\label{indexdef}
{\rm Index}D_u\overline{\partial} = \mu([u]) + \dim L.
\end{equation}
The number $\mu([u])$ is called the Maslov index of $u$.
\begin{defn}
We call
$L$ a {\it monotone Lagrangian submanifold}, if 
there exists a positive number $c$ independent of $u$ such that
\begin{equation}\label{monotoneineq}
\mu([u]) = c \int_{D^2} u^*\omega
\end{equation}
for any $u:(D^2,\partial D^2) \to (X,L)$.
\par
The minimum Maslov number $L$ is by definition:
\begin{equation}
\inf \{ \mu([u]) \mid u : (D^2,\partial D^2) \to (X,L),
\,\, \mu([u]) \ne 0\}.
\end{equation}
\end{defn}
Now we briefly explain the reason why the equality (\ref{neweq122}) holds
in the situation of 
Theorem \ref{monotonecaseoh}.
Let $u_i : \R \times [0,1] \to X$ be a sequence of elements of 
${\mathcal M}(L_1,L_2;a,b;2,E)$.
By taking a subsequence if necessary we may assume that the 
limit looks like either  Figure \ref{zu3} or Figure \ref{zu4}.
The case of Figure \ref{zu3} corresponds to the right hand side of 
(\ref{neweq122}). So it suffices to see that Figure \ref{zu4}
does not occur.
The limit drawn in Figure \ref{zu4} can be regarded as a pair 
of $[u_{\infty}] \in {\mathcal M}(L_1,L_2;a,b;k,E')$
for some $k$ and $E'$ and 
$u'' : (D^2,\partial D^2) \to (X,L_i)$.
Since $u''$ is pseudoholomorphic we have 
$\int_{D^2} (u'')^*\omega > 0$.
Therefore (\ref{monotoneineq}) implies 
$\mu([u'']) > 0$.
Since the minimal Maslov number is greater than $2$ 
we have $\mu([u'']) > 2$.
By (\ref{indexdef}) and index sum formula we can show
$$
k = 2 - \mu([u'']) < 0.
$$
Using this fact we can show 
${\mathcal M}(L_1,L_2;a,b;k,E')$ is an empty set 
when appropriate transversality is satisfied.
This implies that Figure \ref{zu4} does not happen.

\begin{figure}[h]
\centering
\includegraphics[scale=0.2]{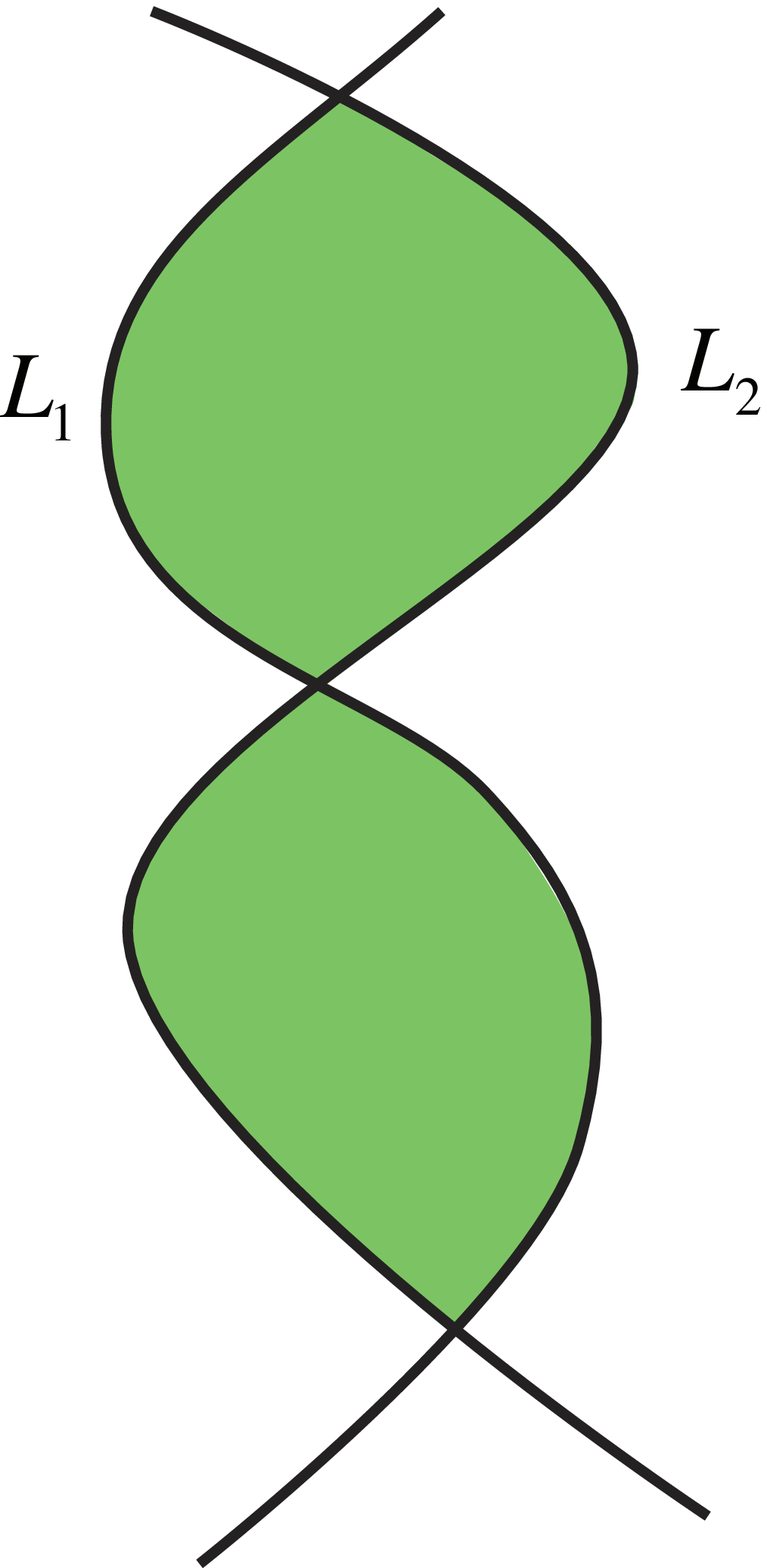}
\caption{End 1}
\label{zu3}
\end{figure}

\begin{figure}[h]
\centering
\includegraphics[scale=0.35]{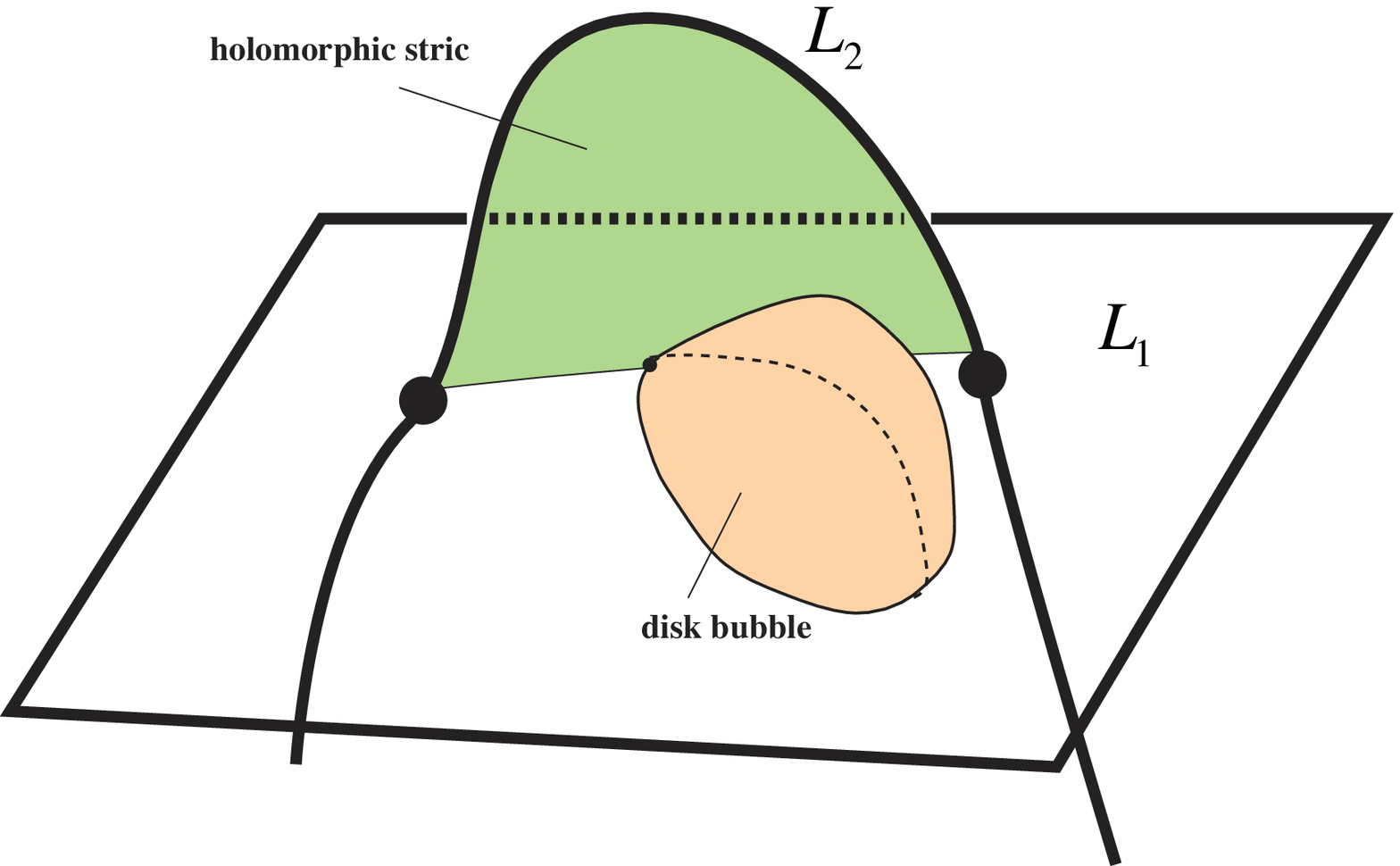}
\caption{End 2}
\label{zu4}
\end{figure}

\begin{rem}
Note in the case when the minimal Maslov number is $2$ 
we may have $k = 2 - 2 = 0$.
Since the (virtual) dimension of 
${\mathcal M}(L_1,L_2;a,b;k,E')$ is $k- 1 = -1$, 
one might imagine that this is enough to show 
that ${\mathcal M}(L_1,L_2;a,b;k,E')$ is an empty set.
However there is a case $a=b$.
In that case ${\mathcal M}(L_1,L_2;a,a;0,0)$ 
consists of one point (the constant map $u$ to $a$).
This is an element of ${\mathcal M}(L_1,L_2;a,a;0,0)$.
So it is nonempty. Note this element  is a 
fixed point of the $\R$ action.
So even though the virtual dimension of 
${\mathcal M}(L_1,L_2;a,a;0,0)$ is $-1$ it can  
still be nonempty.
This is the reason why we need to assume that minimal Maslov 
index is {\it strictly} larger than $2$ in Theorem \ref{monotonecaseoh}.
This point is studied in great detail in \cite[Section 3.6.3 etc.]{fooobook}.
The notion of potential function introduced in \cite[Definition  3.6.33]{fooobook}
is related to this point.
\end{rem}
As we have seen in this section, in the situation when 
the space $R(M;\mathcal P_M)$ is embedded in 
$R(\partial M;\mathcal P_{\partial M})$
we can use its monotonicity to define Lagrangian Floer homology 
and prove Theorem \ref{embedcase}.
\par
In general $R(M;\mathcal P_M)$ is an 
immersed Lagrangian submanifold in $R(\partial M;\mathcal P_{\partial M})$
even after making appropriate perturbation.
However we still have a kind of monotonicity.
\begin{defn}\label{defn4111}
Let $(X,\omega)$ be a symplectic manifold and 
$i_L : \tilde L \to X$ a Lagrangian immersion.
Namely
$i_L$ is an immersion, $\dim \tilde L = \frac{1}{2} \dim X$
and $i_L^*\omega = 0$.
\par
We say $L = (\tilde L,i_L)$ is {\it monotone in the weak sense}
if for each pair $(u,\gamma)$ such that $u : D^2 \to X$ and 
$\gamma : \partial S^1 \to \tilde L$ with 
$u\vert_{\partial D^2} = i_L \circ \gamma$ the equality
\begin{equation}\label{monotoneineq2}
\mu([u]) = c \int_{D^2} u^*\omega
\end{equation}
holds. (Note the Maslov index $\mu([u])$ can be defined in a similar way 
as embedded case.)
\par
The minimal Maslov number is defined in the same way.
\end{defn}
\begin{prop}
Let $(M,\mathcal P_M)$ be a pair of an oriented 3 manifolds with boundary 
and a principal $SO(3)$ bundle on it such that
$w^2(\mathcal P_M)\vert_{\partial M}$ is the fundamental class 
$[\partial M]$.
We also assume Assumption \ref{assum122}.
Moreover we assume $R(M;\mathcal P_M) \to R(\partial M;\mathcal P_{\partial M})$
is an {\it immersion}.
\par
Then $R(M;\mathcal P_M)$ is an immersed monotone Lagrangian submanifold
in the weak sense.
Its minimal Maslov number is congruent to $0$ modulo $4$.
\end{prop}
However we cannot generalize Theorem \ref{monotonecaseoh}
to the case of immersed monotone Lagrangian submanifold
in the weak sense with minimal Maslov number $>2$.
In fact other than those drawn in Figures \ref{zu3} and \ref{zu4}
there exists another type of boundary of the moduli space 
${\mathcal M}(L_1,L_2;a,b;2,E)$, which is drawn in  Figure \ref{zu5}
below.
So (\ref{neweq122}) does not hold in this generality.
\par
Neverthless actually we can still define the right hand side of 
(\ref{form4545}) and can prove the isomorphism (\ref{form4545}).
For this purpose we need various new ideas which was developped 
in Lagrangian Floer theory during 1996-2009.
We describe some of them in the next section.
\begin{figure}[h]
\centering
\includegraphics[scale=0.3]{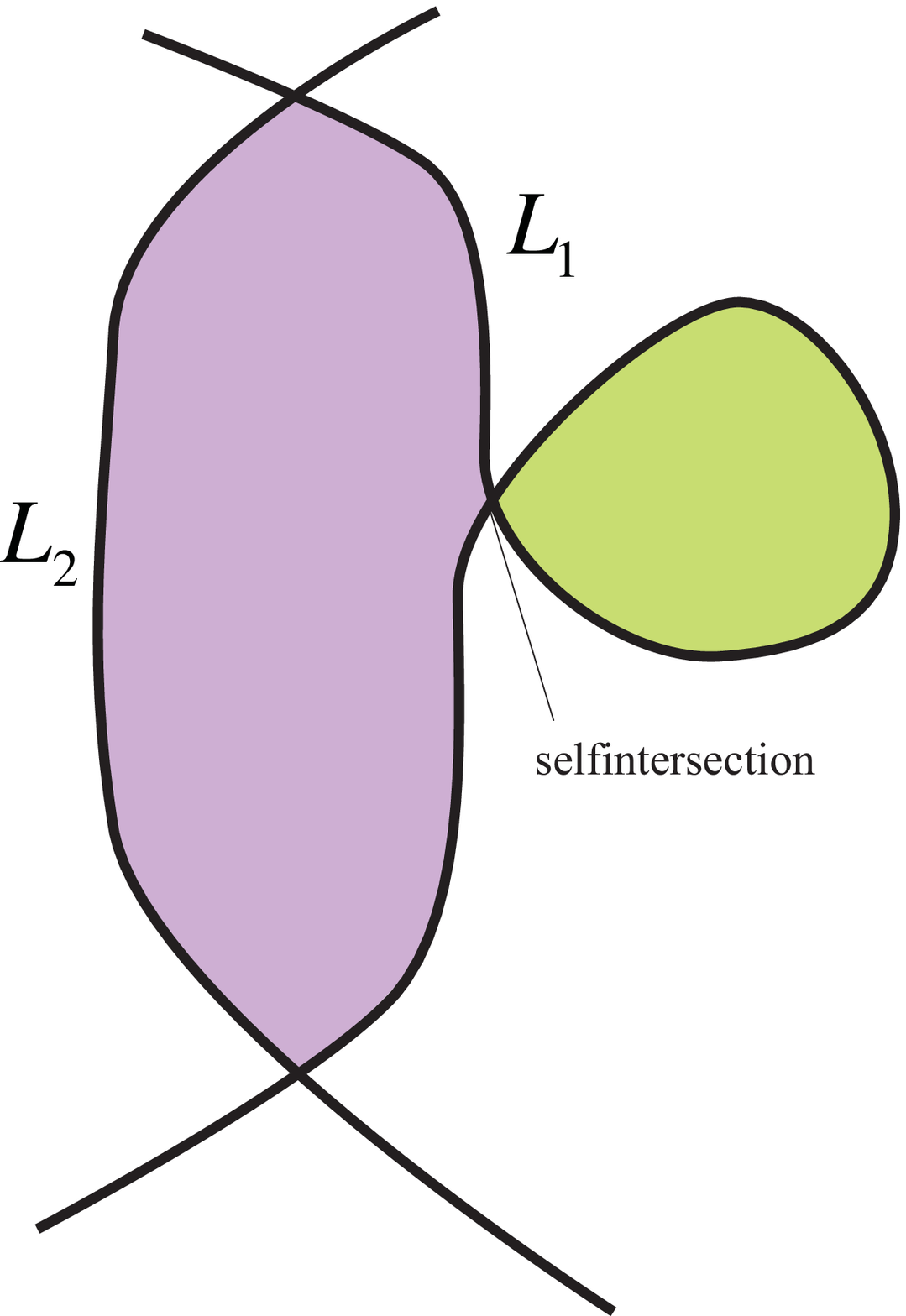}
\caption{End 3}
\label{zu5}
\end{figure}

\section{Biased review of Lagrangian Floer theory II}
\label{section:Lagrangian Floer theory2}
\subsection{Filtered $A_{\infty}$ category}
\label{Ainfinitycate}
This subsection is a brief introduction to filtered $A_{\infty}$ category and algebra.
For more detail, see \cite{AFOOO1,efl,fu4,fooobook,ancher,Ke,Lef,Ly,sei}
and etc..
\par
For the application to gauge theory such as Problem \ref{prob21}
we can work on $\Z$ or $\Z_2$ coefficient as we will explain in Subsection \ref{mainthm}.
However for the general story it is more natural to use universal Novikov ring
which was introduced in \cite{fooobook}.
We first define it below.
Let $R$ be a commutative ring with unit, which we call ground ring.
The reader may consider $R = \Z_2, \Z, \R$ or $\C$.
We take a formal variable $T$ and consider the formal sum
\begin{equation}\label{formalsum51}
\sum_{k=0}^{\infty} a_k T^{\lambda_k}
\end{equation}
such that:
\begin{enumerate}
\item[(NR.1)]
$a_k \in R$.
\item[(NR.2)]
$\lambda_k \in \R_{\ge 0}$.
\item[(NR.3)]
$\lambda_k < \lambda_{k+1}$.
\item[(NR.4)]
$\lim_{k\to \infty} \lambda_k = \infty$.
\end{enumerate}
We call the totality of such formal sum 
(\ref{formalsum51}) the {\it universal Novikov ring} and denote it by $\Lambda_0^R$.
We replace (NR.2) by $\lambda_k \in \R$ (resp. $\lambda_k > \R$) to define 
$\Lambda^R$ (resp. $\Lambda_+^R$.)
$\Lambda_0^R$ and $\Lambda^R$ become rings with unit in an obvious way.
$\Lambda^R_+$ is an ideal of $\Lambda_0^R$.
In case $R$ is a field, $\Lambda^R$ becomes a field, which we call 
(universal) Novikov field.
In case $R$ is a field,  $\Lambda^R_+$ is a maximal ideal of $\Lambda_0^R$.
In fact $\Lambda_0^R/\Lambda_+^R = R$.
\par
The ring $\Lambda^R$ has a filtration $F^{\lambda}\Lambda^R$ which consists of 
(\ref{formalsum51}) with $a_0 \ge \lambda$. It induces 
a filtration $F^{\lambda}\Lambda^R_0$ of $\Lambda^R_0$.
This filtration defines a metric on these rings, by which they become complete.
\par
Hereafter we omit $R$ and write $\Lambda_0$, $\Lambda$, $\Lambda_+$
in place of 
$\Lambda_0^R$, $\Lambda^R$, $\Lambda_+^R$, sometimes.
\begin{defn}
A filtered $A_{\infty}$ category $\mathscr C$  consists of the following objects.
\begin{enumerate}
\item
A set $\frak{Ob}(\mathscr C)$, which is the set of objects.
\item
A graded $\Lambda_0$ module $\mathscr
C(c_1,c_2)$ for each
$c_1,c_2 \in \frak{Ob}(\mathscr C)$.
We call $\mathscr C(c_1,c_2)$ the {\it module of morphisms}.
It is a completion of a free  $\Lambda_0$ module.
(We may consider $\Z$ grading or $\Z_{2N}$ grading for some $N \in \Z_{>0}$.
In our application in  Section \ref{boundary}, we have $\Z_{4}$ grading.
The number $2N$ in our geometric applications is a minimal Maslov number 
explained in the last section.)
\item
For each
$c_0,\dots, c_k \in \frak{Ob}(\mathscr C)$, we are given  operations
($\Lambda_0$ linear homomorphisms)
\begin{equation}\label{strmapop}
{\frak m}_k :  \mathscr C[1](c_0,c_1) \widehat\otimes \cdots \widehat\otimes \mathscr C[1](c_{k-1},c_k) 
\to \mathscr C[1](c_0,c_k),
\end{equation}
of degree $+1$ for $k=0,1,2,\cdots$ and $c_i \in \frak{Ob}(\mathscr C)$,
which preserves the filtration.
We call it {\it structure operations}.
Here $\mathscr C[1](c,c')$ is the degree shift of $\mathscr C(c,c')$. Namely
the degree $k$ part of $\mathscr C[1](c,c')$ is degree $k+1$ part of $\mathscr C(c,c')$.
The symbol $\widehat\otimes$ denotes the $T$-adic completion of the algebraic 
tensor product.
\item
The following $A_{\infty}$ relation is satisfied.
\begin{equation}\label{formula25}
0=\sum_{k_1+k_2=k+1}\sum_{i=0}^{k_1-1}
(-1)^* \frak m_{k_1}(x_1,\dots,x_i,\frak m_{k_2}(x_{i+1},\dots,x_{k_2}),
\dots,x_k),
\end{equation}
where 
$* = i +\sum_{j=1}^i \deg x_j$, 
$x_i \in \mathscr C[1](c_{i-1},c_i)$, $c_0,\dots,c_k \in \frak{Ob}(\mathscr C)$.
\item
We require
$$
\frak m_0(1) \equiv 0 \mod T^{\epsilon},
$$
for some $\epsilon > 0$.
\item
An element ${\bf e}_{c} \in \mathscr C(c,c)$ of degree $0$
is given for each $c \in \frak{Ob}(\mathscr C)$
such that:
\begin{enumerate}
\item If $x_1 \in \mathscr C(c,c')$, $x_2 \in \mathscr C(c',c)$ then
$$
\frak m_2(\text{\bf e}_c,x_1) = x_1,
\quad \frak m_2(x_2,\text{\bf e}_c) = (-1)^{\deg x_2}x_2.
$$
\item
If $k+\ell \ne 1$, 
$x_1 \otimes \dots \otimes x_{\ell} 
\in B_{\ell}\mathscr C[1](a,c)$,
$y_1 \otimes \dots \otimes y_{k} 
\in B_k\mathscr C[1](c,b)$  then
\begin{equation}\label{formidenty}
\frak m_{k+\ell+1}(x_1,\cdots,x_{\ell},\text{\bf e}_c,y_1,\cdots,y_{k}) = 0.
\end{equation}
We call ${\bf e}_{c}$ the {\it unit}.\footnote{In some reference we do not assume unit to exist 
for filtered $A_{\infty}$ category. In this article we assume it to simplify the notation.}
\end{enumerate}
\par
\end{enumerate}
\par
A filtered $A_{\infty}$ category with one object is called a {\it filtered $A_{\infty}$ algebra}.
\par
A filtered $A_{\infty}$ category or algebra is called {\it strict} if $\frak m_0 =0$.
It is called {\it curved} otherwise.
\end{defn}
Note $\frak m_1 : \mathscr C(c,c') \to \mathscr C(c,c')$ is degree one and is regarded as a `boundary operator'.
However in general $\frak m_1 \circ \frak m_1 = 0$ does {\it not hold}.
In fact (\ref{formula25}) implies
\begin{equation}\label{form54}
\frak m_1(\frak m_1(x)) + \frak m_2(\frak m_0(1),x) + (-1)^{\deg x+1} \frak m_2(x,\frak m_0(1)) = 0.
\end{equation}
On the other hand (\ref{form54}) implies that $\frak m_1 \circ \frak m_1 = 0$ 
if  filtered $A_{\infty}$ category $\mathscr C$ is strict.
The algebraic point about the well defined-ness of $\frak m_1$ homology,
which we mentioned above,
is closely related to the geometric problem to define Lagrangian Floer theory,
which we mentioned 
at the end of the last section. We will go back to this point in the next subsection.
\par
An idea introduced in \cite{fooobook} is to deform Floer's boundary operator 
$\frak m_1$ to $\frak m_1^b$ so that $\frak m^b_1 \circ \frak m^b_1 = 0$  
holds.
\begin{defn}\label{defn52}
Let $c \in \frak{Ob}(\mathscr C)$. A {\it bounding cochain} 
(or {\it Maurer-Cartan element}) of $c$ is an element 
$b \in \mathscr C(c,c)$ such that:
\begin{enumerate}
\item
The degree of $b$ is $1$.
\item
$b \equiv 0 \mod \Lambda_+$.
\item
\begin{equation}\label{form55}
\sum_{k=0}^{\infty} \frak m_k(b,\dots,b) = 0.
\end{equation}
\end{enumerate}
Note (2) implies that the left hand side of (\ref{form55}) converges in $T$ adic topology.
\par
We denote by $\widetilde{\mathcal M}(c)$ the set of all bounding cochains.
We say an object $c$ is {\it unobstructed} if $\widetilde{\mathcal M}(c)$
is nonempty.
\end{defn}
\begin{rem}
\begin{enumerate}
\item
We can define appropriate notion of gauge equivalence among elements of 
$\widetilde{\mathcal M}(c)$. (See \cite[Section 4.3]{fooobook}.)  The set 
of all gauge equivalence classes is called 
{\it Maurer-Cartan moduli space} and is written as ${\mathcal M}(c)$.\footnote{It can be regarded 
as a set of rigid points of certain rigid analytic stack.}
\item
In certain situation we may relax the condition Definition \ref{defn52} (2)
and can use a class $b \in \mathscr C(c,c)$ of degreee $1$ 
which satisfies (\ref{form55}). (In such a case the left hand side 
of (\ref{form55}) should be defined carefully.)
We write $\widetilde{\mathcal M}(c;\Lambda_+)$ in place of $\widetilde{\mathcal M}(c)$ 
when we want to clarify that we consider only the elements satisfying  Definition \ref{defn52} (2).
\end{enumerate}
\end{rem}
\begin{defnlem}
Let $c, c' \in \frak{Ob}(\mathscr C)$ and $b \in \widetilde{\mathcal M}(c)$,
$b' \in \widetilde{\mathcal M}(c')$.
We define 
\begin{equation}
\frak m_1^{b,b'} : \mathscr C(c,c') \to \mathscr C(c,c')
\end{equation}
by the formula:
\begin{equation}
\frak m_1^{b,b'}(x) = \sum_{k,\ell=0}^{\infty}
\frak m_{k+\ell+1}(\underbrace{b,\dots,b}_{k},x,\underbrace{b',\dots,b'}_{\ell}).
\end{equation}
\par
(\ref{formula25}) and (\ref{form55}) imply
$$
\frak m_1^{b,b'} \circ \frak m_1^{b,b'} = 0.
$$
We define 
\begin{equation}
HF((c,b),(c',b')) = \frac{{\rm Ker}(\frak m_1^{b,b'} : \mathscr C(c,c') \to \mathscr C(c,c')
)}{{\rm Im}( \frak m_1^{b,b'} : \mathscr C(c,c') \to \mathscr C(c,c')
)} 
\end{equation}
and call it the Floer cohomology of $(c,b)$ and $(c',b')$. 
\end{defnlem}
We can deform $\frak m_k$ in the same way as follows.
Hereafter we write
$
\frak m_{k+\ell+1}(b^k,x,(b')^{\ell})
$
etc. in place of 
$\frak m_{k+\ell+1}(\underbrace{b,\dots,b}_{k},x,\underbrace{b',\dots,b'}_{\ell})$
etc..
\begin{defnlem}
Let $\mathscr C$ be a curved filtered $A_{\infty}$
category. We define a strict filtered $A_{\infty}$ category 
$\mathscr C'$ as follows.
\begin{enumerate}
\item
An object of $\mathscr C'$ is a pair $(c,b)$ where $c \in 
\frak{OB}(\mathscr C)$
and $b \in \widetilde{\frak M}(c)$.
\item
If $(c,b),(c',b')$ are objects of $\mathscr C'$ then
$
\mathscr C'((c,b),(c',b')) = \mathscr C(c,c')
$
by definition.
\item\label{form5959}
If $(c_i,b_i) \in \frak{OB}(\mathscr C')$ for $i=0,\dots,k$
and $x_i \in \mathscr C'((c_{i-1},b_{i-1}),(c_i,b_i)) = \mathscr C(c_{i-1},c_i)$
for $i=1,\dots,k$. Then we define the structure operations $\frak m_k^{(b_0,\dots,b_k)}$ 
of $\mathscr C'$ as follows.
\begin{equation}\label{form5959}
\frak m_k^{(b_0,\cdots,b_k)}(x_1,\cdots,x_k)
= \sum_{\ell_0,\cdots,\ell_k}
\frak m_{k+\ell_0+\cdots+\ell_k}(b_0^{\ell_0},x_1,b_1^{\ell_1},\cdots,
b_{k-1}^{\ell_{k-1}},x_k,b_k^{\ell_k}).
\end{equation}
\par
We call $\mathscr C'$ the {\it associated strict category} to $\mathscr C$.
\end{enumerate}
\end{defnlem}
Note 
$$
\frak m_0^{(b)}(1) = \sum_{k=0}^{\infty} \frak m_k(b,\dots,b) = 0
$$
by (\ref{form55}).
We omit the proof that the structure operations 
(\ref{form5959}) satisfies the relation (\ref{formula25}), 
which is an easy calculation.

\subsection{Immersed Lagrangian submanifold and its Floer homology}
\label{immersed}

Let $i_L : \tilde L \to X$ be an $n$ dimensional immersed submanifold of 
a symplectic manifold $X$ of dimension $2n$.
We say $L = (\tilde L,i_L)$ is an 
{\it immersed Lagrangian submanifold}
if $i_L^*\omega = 0$.
\par
\begin{defn}\label{defn56}
We say that $L$ has {\it clean self-intersection} if 
the following holds.
\begin{enumerate}
\item
The fiber product $\tilde L \times_X \tilde L$ is a 
smooth submanifold of $\tilde L \times \tilde L$.
\item
For each $(p,q) \in \tilde L \times_X \tilde L$
we have
$$
\{(V,W) \in T_{p}\tilde L \oplus T_{q}\tilde L
\mid di_L(V) = di_L(W)\}
=
T_{(p,q)}(\tilde L \times_X \tilde L).
$$
\end{enumerate}
We say $L$ has {\it transversal self-intersection}
if 
$$
(\tilde L \times_X \tilde L) 
\setminus \tilde L 
$$
is a finite set. (Note the fiber product 
$\tilde L \times_X \tilde L$ contains the 
diagonal $\cong \tilde L$.)
\par
A finite set of immersed Lagrangian submanifolds $\{L_i \mid
i=1,\dots,N\}$ is said a {\it clean collection} 
(resp. {\it transversal collection}) if 
the disjoint union $\bigcup_{i=1}^N L_i$ has 
clean self-intersection (resp. has transversal self-intersection).
\end{defn} 
Lagrangian Floer theory in 
\cite{fooobook, fooobook2} associates a filtered $A_{\infty}$
algebra to an embedded Lagrangian submanifold $L$.
This $A_{\infty}$ algebra as a $\Lambda_0$ modle is taken 
to be the cohomology group of $L$ or any of its chain model.
Namely it defines a homomorphism
$
\frak m_k : H(L;\Lambda_0)^{\otimes k} \to H(L;\Lambda_0)
$
satisfying (\ref{formula25}).
It also associates a filtered $A_{\infty}$ 
category ${\frak{Fuk}}(\frak L)$ to a transversal collection of embedded Lagrangian 
submanifolds $\frak L = \{L_i \mid i=1,\dots,N\}$
as follows.
\begin{enumerate}
\item[(L.C1)]
The set of object is $\frak L$.
\item[(L.C2)]
For $L_{i},L_{j} \in \frak L$, the module of morphisms, which we write 
$CF(L_{i},L_{j})$, is defined as follows:
\begin{enumerate}
\item
If $i \ne j$ then it is a free $\Lambda_0$ module whose 
basis is identified with $L_i \cap L_j$.
\item
If $i=j$ then $CF(L_{i},L_{i}) = H(L_i;\Lambda_0)$.
\footnote{We may also take $CF(L_{i},L_{i})$ (in principle any) chain model 
of the cohomology of $L_i$. For example in \cite{fooobook, fooobook2, 
ancher,fu3}
the singular chain complex is used. In 
\cite{AFOOO1,efl,foootech2,foootech22} etc. 
de-Rham complex is used.}
\end{enumerate} 
\item[(L.C3)]
The structure operations (\ref{strmapop}) is defined by 
using the moduli space of pseudoholomorphic $k+1$ gons.
\end{enumerate}
Figure \ref{zu6} below is the moduli space of $k+1$ gons, which 
calculate the coefficient of $[p]$ in 
$
\frak m_7(x_1,\dots,x_7).
$
\begin{figure}[h]
\centering
\includegraphics[scale=0.3]{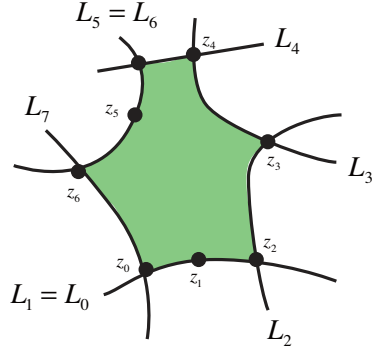}
\caption{$\langle\frak m_7(x_1,\dots,x_7),x_0\rangle$}
\label{zu6}
\end{figure}
\par
Akaho-Joyce \cite{AJ} generalized this story to the immersed case as follows.
Let $\frak L$ be a transversal collection of immersed Lagrangian 
submanifolds $\frak L = \{L_i \mid i=1,\dots,N\}$.
Then we have a 
filtered $A_{\infty}$ category satisfying 
(L.C1)-(L.C3), except (L.C2) (2) is replaced by
\begin{enumerate}
\item[(2')]
If $i=j$ then 
\begin{equation}\label{immFl}
CF(L_{i},L_{i}) 
= H(L_i;\Lambda_0) 
\oplus \bigoplus \Lambda_0 (p,q)
\end{equation}
where the direct sum is taken over all
$(p,q)$
such that $i_{L_i}(p) = i_{L_i}(q)$ and 
$p\ne q$.
\par
We say $\bigoplus \Lambda_0 (p,q)$ the {\it switching part} 
and $H(L_i;\Lambda_0)$ the {\it diagonal part} of 
$CF(L_{i},L_{i})$.
\end{enumerate}
We remark that $(p,q)$ is an {\it ordered} pair. Namely
$(p,q) \ne (q,p)$. So we associate extra two generators 
to each self-intersection.
\par
The definition of structure operation including 
the generator $(p,q)$ is similar and by using 
the moduli space drawn below.
\begin{figure}[h]
\centering
\includegraphics[scale=0.3]{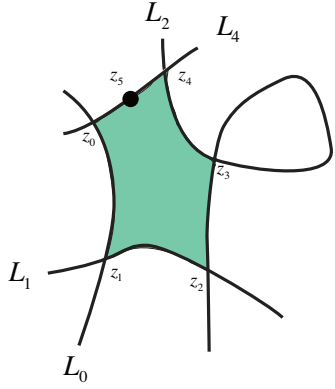}
\caption{$\langle\frak m_7(x_1,\dots,x_5),x_0\rangle$}
\label{zu7}
\end{figure}
Note the 3rd marked point in the figure corresponds 
$(p,q)$ which is the extra generator appearing in 
item (2)' above. Those generators correspond to the switchings at the 
boundary curve. We remark if $x = i_{L_i}(p) = i_{L_i}(q)$
is the self intersection point of $L_i$, 
the intersection $L_i \cap U$ of $L_i$ with a small neighborhood $U$ of $x$ 
consists of two smooth $n$ dimensional submanifolds.
They are image of neighborhoods of $p$ and of $q$ respectively.
Switching at 3rd marked point means that the boudary value 
was on one of the components of  $L_i \cap U$
for $z \in \partial D^2$ $z<z_3$ and 
on the other component of $L_i \cap U$
for $z \in \partial D^2$ $z>z_3$.
There are two different ways how the 
switching occurs. One is the switching from the component containing $p$
to the component containing $q$, and the other is the switching from the component containing $q$
to the component containing $p$.
Those two different ways of switching correspond to the two generators 
appearing in (\ref{immFl}).
\par

\subsection{The case of immersed Lagrangian submanifold
which is monotone in the weak sense}
\label{subsec:weakmonotone}

Note (\ref{form54}) shows that the term $\frak m_0(1)$\footnote{which is 
called curvature sometimes} causes the 
problem to define Floer homology.
Namely if $\frak m_0(1) = 0$ then 
$\frak m_1 \frak m_1 = 0$.
In other words we are looking for an appropriate 
element $b$ of $CF(L_i,L_i)$
such that $\frak m^b_0(1) = 0$.
We called such $b$ a {\it bounding cochain}.
The progress we made recently in our study 
of 3+2 dimensional Donaldson-Floer theory, 
is that we can now prove the existence of such $b$ 
in the situation appearing in it.
We explain it in Section \ref{boundary}.
Before doing so we go back to the situation of 
Subsection \ref{subsec:monotone}.
\par
$\frak L$ to a transversal collection of immersed Lagrangian 
submanifolds $\frak L = \{L_i \mid i=1,\dots,N\}$.
\begin{prop}\label{prop5757}
Suppose each of $L_i$ are 
immersed monotone Lagrangian submanifolds
in the weak sense and its minimal Maslov index 
is greater than $2$. Then the element  
$$
\frak m_0(1) \in CF(L_i,L_i)
$$
lies in the switching part.
\end{prop}
\begin{proof}\label{prop57}
The diagonal part of $\frak m_0(1)$ is 
defined as the (virtual) fundamental class of
the moduli space of the pair 
$((D^2,\partial D^2),z_0),u)$
where $u :  (D^2,\partial D^2) \to (X,L_i)$ is a 
pseudoholomorphic curve such that $\partial D^2 \to L_i$ 
lifts to $\tilde L_i$.
Therefore by assumption the (virtual) dimension of such 
modulie space is $\mu(\beta) + n - 3 +1$. 
Since $\mu(\beta) > 2$ by assumption 
the dimension is greater than $n$. 
Hence it lines in $H_*(L_i)$ with $* > n = \dim L_i$ and is $0$.
\end{proof}
By Proposition (\ref{prop5757}) and (\ref{form54})
the right hand side of
$$
\frak m_1\frak m_1(x) 
= \frak m_2(x,\frak m_0(1)) + \frak m_2(\frak m_0(1),x)
$$
is defined by the moduli space as in Figure \ref{zu5}.
In other words this formula is an algebraic way 
to explain the difficulty to define the 
Floer homology of a pair of immersed Lagrangian submanifolds 
which is monotone but immersed.
(We explained it in Subsection \ref{boundary} in a geometric way.)
\par
We also have the following:
\begin{prop}\label{prop5758}
In the situation of Proposition \ref{prop57} the sum appearing in the definition 
of structure operations are all finite sum and the  
construction works over the ground ring $\Z_2$ (or $\Z$ in the case when we can find orientations 
of the moduli space which is compatible at the boudnaries.).
\end{prop}
The proof of the first half is similar to the proof of Proposition 
\ref{prop5757}. The proof of the second half is similar to 
\cite{fooo:overZ}.

\section{Instanton Floer homology for 3 manifolds with boundary}
\label{boundary}

\subsection{Main theorems}
\label{mainthm}

The next two theorems are the main results explained in this article.
We assume the following for the simplicity of 
the statement.
Let $M$ be a 3 manifold with boundary $\Sigma = \partial M$
and let $\mathcal P_M$ be a principal $SO(3)$ bundle.
We denote $\mathcal P_{\Sigma}$ the restriction of 
$\mathcal P_M$ to $\Sigma$. We assume $w^2(\mathcal P_{\Sigma}) 
= [\Sigma]$. (Assumption \ref{assumption11} (1).)

\begin{assump}\label{ass61}
We assume Assumption \ref{assum122}.
Moreover we assume that the restriction map defines an immersion
$R(M;\mathcal P_M) \to R(\Sigma;\mathcal P_{\Sigma})$.
\end{assump}
We remark that we can reduce the general case to the case when this assumption 
is satisfied by perturbing the equation  $F_a = 0$ in the same way as 
\cite{Don2, fl1, He}.

\begin{thm}\label{them61}
We assume Assumption \ref{ass61}. Then
the object $R(M;\mathcal P_M)$ is unobstructed.
Namely there exists a bounding cochain $b_M$ 
in $CF(R(M;\mathcal P_M))$.
\par
Moreover we can find $b_M$ in a canonical way.
In other words, its gauge equivalence class is an invariant 
of $(M;\mathcal P_M)$.
\end{thm}
\begin{rem}
\begin{enumerate}
\item
We omit the definition of gauge equivalence between 
bounding cochains. See \cite[Definition 4.3.1]{fooobook}.
\item
We can remove Assumption \ref{ass61} in Theorem \ref{them61} 
by perturbing the equation $F_a = 0$ (which defines 
$R(M;\mathcal P_M)$) on a compact subset of $M$.
The same remark applies to the next theorem.
\item
We can also prove that $b_M$ is supported in the 
switching components.
\end{enumerate}
\end{rem}
\begin{thm}\label{thm64}
Suppose $(M_1,\mathcal P_{M_1})$ and $(M_2,\mathcal P_{M_2})$
satisfy Assumption \ref{assum122}.
We also assume
$$
\partial(M_1,\mathcal P_{M_1}) = (\Sigma,\mathcal P_{\Sigma}),
\qquad
\partial(M_2,\mathcal P_{M_2}) = (-\Sigma,\mathcal P_{\Sigma}).
$$
%and that $R(M_1;\mathcal P_{M_1})$ is transversal to 
%$R(M_2;\mathcal P_{M_2})$.
We glue $(M_1,\mathcal P_{M_1})$ and $(M_2,\mathcal P_{M_2})$
along their boundaries to obtain
$(M,\mathcal P_M)$.
\par
Then we have an isomorphism:
\begin{equation}\label{iso61}
HF((R(M_1;\mathcal P_{M_1}),b_{M_1}),
(R(M_2;\mathcal P_{M_2}),b_{M_2}))
\cong
HF(M,\mathcal P_{M}) \otimes \Lambda_0^{\Z_2}.
\end{equation}
\end{thm}
Note the Lagrangian Floer homology in the left hand side has 
$\Lambda_0^{\Z_2}$ coefficient.
However actually it can be defined over $\Z_2$ coefficient.
In fact we have:
\begin{prop}
In the situation of Theorem \ref{them61}
we may choose the bounding cochain $b_M$ so that it is 
entirely in the switching part.
\end{prop}
Then using Proposition \ref{prop5758} we can define 
Floer homology 
$$
HF((R(M_1;\mathcal P_{M_1}),b_{M_1}),
(R(M_2;\mathcal P_{M_2}),b_{M_2}))
$$ 
over $\Z_2$ coefficient and  (\ref{iso61}) holds over $\Z_2$ 
coefficient.

Theorems \ref{them61} and \ref{thm64} will be proved in a 
forthcoming joint paper with Alikebar Daemi.

\subsection{A possible way to obtain bounding cochain $b_M$ directly 
from moduli spaces of ASD-connections}
\label{analysisbounding}

In this subsection we explain a conjecture which provides a way 
to obtain the bounding cochain $b_M$ in Theorem \ref{them61},
directly by counting the order of certain moduli space of 
ASD-connections.
The author is unable to prove this conjecture, which looks 
rather hard. 
The proof of Theorem \ref{them61} is performed in a different way
as we will sketch in Subsections \ref{cyclic} and \ref{exist}.
\par
Let $(M,\mathcal P_M)$ be as in Theorem \ref{them61}.
We consider the interior of $M$, that is $M \setminus \partial M$.
We write it $M$ for simplicity of notation in this subsection.
We choose its Riemannian metric such that 
$M$ minus a compact set is isometric to the direct product
$\Sigma \times (0,\infty)_t$. We use $t$ as the coordinate of 
$(0,\infty)_t$. We put the suffix  $t$ to clarify this point.
We take direct product
$M \times \R_{\tau}$ and use $\tau$ as the coordinate of $\R$ factor
and consider a principal $SO(3)$ bundle $\mathcal P_M \times \R_{\tau}$.

\begin{defn}
We consider the set of all connections $A$ of 
$\mathcal P_M \times \R_{\tau}$ with the following properties.
\begin{enumerate}
\item
$F_{A}^+ = 0$. Namely $A$ is an Anti-Self-Dual connection.
\item
We have
\begin{equation}\label{form62}
\int_{M\times \R_{\tau}} \Vert F_{A}\Vert^2\,\, {\rm Vol}_M d\tau
= E < \infty.
\end{equation}
\end{enumerate}
We denote by $\widetilde{\mathcal M}(M \times \R_{\tau};\mathcal P_M \times \R_{\tau};E)$
the set of all gauge equivalence classes of such $A$.
(Here $E$ is as in (\ref{form62}).)
\par
We can define an $\R_{\tau}$ action on 
$\widetilde{\mathcal M}(M \times \R_{\tau};\mathcal P_M \times \R_{\tau};E)$
by the translation of $\R_{\tau}$ direction.
We denote by
${\mathcal M}(M \times \R_{\tau};\mathcal P_M \times \R_{\tau};E)$
the quotient space of this $\R$ action.
\end{defn}

\begin{conj}
We assume Assumption \ref{ass61}.
Moreover we assume that $R(M;\mathcal P_M)$ has transversal selfintersection.
Then we can find an $\R_{\tau}$ invariant perturbation supported on a compact subset
of $M$ 
so that the following holds.
\begin{enumerate}
\item
$\mathcal M(M \times \R_{\tau};\mathcal P_M \times \R_{\tau};E)$
becomes a finite dimensional manifold.
\item
We can compactify it so that the singularity of the compactification 
has codimension $4$.
\item
Any element of $\mathcal M(M \times \R_{\tau};\mathcal P_M \times \R_{\tau};E)$ is 
gauge equivalent to a connection $A$ with the following properties.
\begin{enumerate}
\item There exist $[a],[b] \in R(M;\mathcal P_M)$ such that
$$
\lim_{\tau \to -\infty} A\vert_{M\times \{\tau\}} =a,
\qquad
\lim_{\tau \to +\infty} A\vert_{M\times \{\tau\}} =b.
$$
\item
There exists $[\alpha] \in R(\Sigma;\mathcal P_{\Sigma})$
such that the following holds for any $\tau$.
$$
\lim_{t\to \infty} A\vert_{\Sigma\times \{(t,\tau)\}} = \alpha.
$$
Note we require $\alpha$ to be independent of $\tau$.
\item
In particular the restrction of $a$ and $b$ to $\Sigma$ 
are both gauge equivalent to $\alpha$.
\end{enumerate}
\item
If the dimension of $\mathcal M(M \times \R_{\tau};\mathcal P_M \times \R_{\tau};E)$
is zero then $[a] \ne [b]$ in item (3).
\item
We put 
${\rm Switch} = \{([a],[b]) \mid [a] \ne [b],\,\,
[a\vert_{\Sigma}] = [b\vert_{\Sigma}]\}$.
For each $([a],[b]) \in {\rm Switch}$ let 
$c([a],[b])$ be the number of elements $[A]$
as in item (3) such that it is in the zero dimensional 
component. Then the sum
$$
\sum_{([a],[b]) \in {\rm Switch}}c([a],[b]) ([a],[b])
\in CF(R(M;\mathcal P_M),R(M;\mathcal P_M))
$$
is the bounding cochain $b_M$  in Theorem \ref{them61}.
\end{enumerate}
\end{conj}
This conjecture is difficult to prove.
It seems that item (3) is a kind of 
Gauge theory analogue of a result by 
Bottman \cite{bott}.

\subsection{Right $A_{\infty}$ module and cyclic element}
\label{cyclic}

In this and the next subsections we explain a way to go around
the difficult analysis to study the moduli space 
$\mathcal M(M \times \R_{\tau};\mathcal P_M \times \R_{\tau};E)$
but use an algebraic lemma to obtain the bounding cochain $b_M$.
To state it we need a notations.

\begin{defn}
Let $(C,\{\frak m_k\})$ be a filtered $A_{\infty}$ algebra,
which may be curved.
A {\it right filtered $A_{\infty}$ module}
over $(C,\{\frak m_k\})$ is $(D,\{\frak n_k\})$
such that:
\begin{enumerate}
\item
$D$ is a graded $\Lambda_0$ module which is a completion of free $\Lambda_0$ module.
\item 
For $k=0,1,2,\dots$, 
$$
\frak n_k : D \otimes \underbrace{C[1] \otimes \cdots \otimes C[1]}_k \to D
$$
is a $\Lambda_0$ module homomorphism which preserves filtration
and has degree $1$. Here $C[1]^d = C^{d+1}$ by definition.
\item
The next relation holds for $y \in D$ and $x_1,\dots,x_k \in C$.
\begin{equation}\label{form63}
\aligned
0= &\sum_{\ell=0}^{k}  
\frak n_{k-\ell}(\frak n_{\ell}(y;x_1,\dots,x_{\ell}), x_{\ell+1},\dots,x_k))\\
&+\sum_{i=-1}^{k}\sum_{j=i}^{k} (-1)^*\frak n_{k-j+i+1}(y;x_1,\dots,\frak m(x_{i+1},\dots,x_j),
x_{j+1},\dots,x_k),
\endaligned
\end{equation}
where $* = \deg y + \deg x_1 + \dots + \deg x_i$.
\end{enumerate}
\end{defn}
Let $(D,\{\frak n_k\})$ be a right filtered $A_{\infty}$ module 
over $(C,\{\frak m_k\})$ and $b$ a bounding cochain of
$(C,\{\frak m_k\})$. 
We define
$d^b : D \to D$ by
\begin{equation}\label{form64}
d^b(y) = \sum_{k=0}^{\infty} \frak  n_k(y;b,\dots,b).
\end{equation}
It is easy to see from (\ref{form63}) that
$d^b\circ d^b = 0$.
\begin{defn}
Let $(D,\{\frak n_k\})$ be a right filtered $A_{\infty}$ module 
over $(C,\{\frak m_k\})$.
An element ${\bf 1} \in D$ is said to be a 
{\it cyclic element} if the following holds.
\begin{enumerate}
\item
The map $x \mapsto \frak n_1({\bf 1},x)$
is a $\Lambda_0$ module isomorphism: $C \to D$.
\item
$\frak n_0({\bf 1}) \equiv 0 \mod \Lambda_+$.
\end{enumerate}
\end{defn}
\begin{defn}
Let $G$ be a submonoid of $\R$ which is discrete.
\par
Let $C$ be a completion of a free $\Lambda_0$ module.
An element $x$ of $C$ is said to be 
{\it $G$-gapped} if it is of the form
$$
x = \sum_{i,j} a_{i,j} T^{\lambda_i} e_j
$$
where $a_{i,j} \in R$ (the ground ring), $\lambda_i \in G \subset \R_{\ge 0}$,
and $e_j$ is a basis of $C$.
\par
Suppose $C_i$ ($i=1,2$) are completions of free $\Lambda_0$ modules.
A filtered $\Lambda_0$ module homomorphism from $C_1$ to $C_2$ 
is said to be {\it $G$-gapped} if it sends $G$-gapped elements to 
$G$-gapped elements.
\par
A filtered $A_{\infty}$ algebra (resp. category, module) 
are said to be {\it $G$-gapped} if all of its structure operations 
are $G$-gapped.
\end{defn}
The filtered $A_{\infty}$ categories we obtain in Lagrangian 
Floer theory are always $G$-gapped for some $G$.

\begin{prop}\label{prop611}
Let $(D,\{\frak n_k\})$ be a right filtered $A_{\infty}$ module 
over $(C,\{\frak m_k\})$.
We assume that they are $G$-gappped.
Let ${\bf 1} \in D$ be a 
cyclic element, which is also $G$-gapped.
\par
Then there exists uniquely a $G$-gapped element $b$ of $C$ such that:
\begin{enumerate}
\item $b$ is a bounding cochain of $C$.
\item
\begin{equation}\label{form65}
d^b({\bf 1}) = 0.
\end{equation}
Here $d^b$ is as in (\ref{form64}).
\end{enumerate}
\end{prop}
The proof is actually easy. We regard (\ref{form65}) 
as an equation for $b$. We can solve it by induction on 
energy filtration. (We use $G$-gapped-ness here so that the 
filtration is parametrized by a discrete set.)
Then using $d^b(d^b({\bf 1})) = 0$, we can show that $b$ is 
a bounding cochain.
See \cite[Proposition 3.5]{fu7} for the proof.
\par
We remark that we do not assume that $C$ is unobstructed.
Namely a priori there may not exist bounding cochain.
In other words, we can use Proposition \ref{prop611} to show the existence 
of bounding cochain. This turn out to be a useful tool 
to prove such existence.
We remark that we are unable to define Lagrangian Floer homology
unless we have some bounding cochain. So proving the  existence 
of bounding cochain is a crucial step for various applications of 
Lagrangian Floer theory.

\subsection{Existence of bounding cochain}
\label{exist}

In this subsection we show an outline of the way 
how we use Proposition \ref{prop611} to prove Theorem \ref{them61}.
\par
To put the discussion in an appropriate perspective, we consider the following situation.
Let $(M;\mathcal P_M)$ be as in Theorem \ref{them61}. 
Let $L$ be an immersed Lagrangian submanifold of 
$(\Sigma;\mathcal P_{\Sigma})$.
We assume that $\{R(M;\mathcal P_M), L\}$ is a clean 
collection in the sense of
Definition \ref{defn56}.
\par
We put
\begin{equation}\label{eq6666}
CF((M;\mathcal P_M);L)
=
C_*(R(M;\mathcal P_M) \times_{(\Sigma;\mathcal P_{\Sigma})} L)
\,\,\hat\otimes\,\,
\Lambda_0.
\end{equation}
Here $R(M;\mathcal P_M) \times_{(\Sigma;\mathcal P_{\Sigma})} L$ is the
fiber product of two immersed Lagrangian submanifolds
and is a smooth manifold by our assumption.
$C_*(R(M;\mathcal P_M) \times_{(\Sigma;\mathcal P_{\Sigma})} L)$ is 
certain chain model of the homology group of  this manifold.
$\hat\otimes$ is the completion of algebraic tensor product.
\par
Note $CF((M;\mathcal P_M);L)$ as a $\Lambda_0$ module 
is the same as the underlying $\Lambda_0$ 
module $CF(R(M;\mathcal P_M);L)$ of the chain complex 
which we use to define Floer homology $HF(R(M;\mathcal P_M);L)$.
We recall that we associated to $L$ a 
filtered  $A_{\infty}$ algebra (as in Subsection \ref{immersed}).
We write it $CF(L)$.
\begin{thm}\label{thm612}
On $CF((M;\mathcal P_M);L)$, there exists a structure of right filtered 
$A_{\infty}$ module over $CF(L)$.
\end{thm}
We use the next proposition together with Theorem \ref{thm612}
to prove Theorem \ref{them61}.
\par
We take $L = R(M;\mathcal P_M)$.
Then by definition
$$
CF((M;\mathcal P_M);L)
= 
C_*(R(M;\mathcal P_M) \times_{R(\Sigma;\mathcal P_{\Sigma})}R(M;\mathcal P_M)) \,\,\hat\otimes\,\, \Lambda_0.
$$
We remark 
$R(M;\mathcal P_M)$ is an open submanifold of 
$R(M;\mathcal P_M) \times_{R(\Sigma;\mathcal P_{\Sigma})}R(M;\mathcal P_M))$.
We denote by ${\bf 1} \in CF((M;\mathcal P_M);R(M;\mathcal P_M))$ 
the differential $0$ form which is $1$ on $R(M;\mathcal P_M)$
and is zero on other part.
\begin{prop}\label{prop613}
${\bf 1} \in CF((M;\mathcal P_M);R(M;\mathcal P_M))$ 
is a cyclic element of the right filtered
$A_{\infty}$ module $CF((M;\mathcal P_M);R(M;\mathcal P_M))$ 
over $CF(R(M;\mathcal P_M))$.
\end{prop}
Theorem \ref{them61} is an immediate consequence of 
Proposition \ref{prop611}, Theorem \ref{thm612} and 
Proposition \ref{prop613}.
\par
In the rest of this subsection we briefly explain the idea of 
the proof of Theorem \ref{thm612}.
The idea is to use Lagrangian submanifold 
$L$ as a boundary condition for an ASD-equation on 
$M \times \R$.
(It appeared in \cite{fu1}  in the year 1992.
It is ellaborated in \cite{fu2}.
Actually this is the motivation of the author when he 
introduced the notion of $A_{\infty}$ category in the study 
of gauge theory and of symplectic geometry.)
\par
We put a metric on $M$ such that it is of product 
type $\Sigma \times (-1,0]$ near the neighborhood 
of the boundary $\partial M = \Sigma$.
(Note this metric is different from one we used in 
Subsection \ref{analysisbounding}.
In Subsection \ref{analysisbounding} we take a 
Riemannian metric on $M \setminus \Sigma$ 
which is isometric to $\Sigma \times (0,\infty)$ 
outside a compact set.)
\par
We then consider the product 
and $M \times \R_{\tau}$ and 
a connection $A$ on it such that:
\begin{enumerate}
\item
$A$ is an ASD-connection. Namely $F_A^+ = 0$.
\item
$$
\int_{M \times \R_{\tau}}\Vert F_A\Vert^2 {\rm Vol}_M d\tau
< \infty.
$$
\item
There exists $(a,a'), (b,b') \in R(M;\mathcal P_M)
\times_{R(\Sigma;\mathcal P_{\Sigma})} L$
such that
$$
\lim_{\tau \to -\infty} A\vert_{M\times \{\tau\}} = a,
\quad 
\lim_{\tau \to +\infty} A\vert_{M\times \{\tau\}} = b.
$$
\item
On $\partial M \times \R_{\tau}$ the connection $A$ satisfies a
boundary condition determined by the Lagrangian submanifold $L$.
\end{enumerate}
\begin{rem}
The way how we write Condition (4) above is not precise.
There are three different ways known to set this boundary conditions
at the stage of 2016.
(One which the author proposed in  \cite{fu1} in the year 1992 does not 
seem to give a correct moduli space.)
\begin{enumerate}
\item[(i)]
The method to use a Riemannian metric which is 
degenerate near the boundary.
This was introduced by the author in \cite{fu3} in 1997.
\item[(ii)]
The method to require that $A$ is flat on 
each $\Sigma \times \{\tau\}$ and its gauge equivalence
class is in $L$.
This was introduced and used by Salamon-Wehrheim \cite{Sawe,We,We2}
at the beginning of 21st century. 
\item[(iii)]
Using pseudoholomorphic curve equation near the boundary 
and a matching condition.
This is introduced by Lipyanskiy \cite{Li} around 2010.
\end{enumerate}
The method (i), (iii) both can be used for our purpose.
The compactness and removable singularity results 
which are needed for our purpose are proved 
in  \cite{fu3} and in \cite{Li}.
The detail of the Fredholm theory is not yet written.
\par
The method (ii) works for our purpose if 
$R(M;\mathcal P_M)$ and $L$ are both embedded and monotone.
In such a case \cite{Sawe} gives a proof of Theorem \ref{thm612}.
On the other hand it seems difficult to generalize 
this method beyond the case when $R(M;\mathcal P_M)$ is embedded,
by the reason explained in \cite[Section 6]{fu7}.
\end{rem}
\par
We consider the pair $(A,\vec z)$ where $A$ is a 
connection on $M \times \R_{\tau}$ satisfying 
the above conditions (1)(2)(3)(4) and $\vec z = (\tau_1,\dots,\tau_k)
\in \R_{\tau}$ with $\tau_1 < \dots < \tau_k$.
We denote the totality of (the gauge equivalence class of) such pair 
by $\mathcal M((M;\mathcal P_M),L)$.
We use the boundary value of $A$ at $\Sigma \times \{\tau_i\}$
to define evaluation maps
$$
{\rm ev} = ({\rm ev}_1,\dots,{\rm ev}_k) : \mathcal M((M;\mathcal P_M),L) \to 
(\tilde L \times_{R(\Sigma;\mathcal P_{\Sigma})} \tilde L)^k.
$$
We consider also the case when we switch at the 
marked point $z_i$. So the target space is as above.
Using asymptotic limit as $\tau \to \pm\infty$ we obtain
$$
{\rm ev}_{\pm\infty} :
\mathcal M((M;\mathcal P_M),L)
\to 
R(M;\mathcal P_{M}) \times_{R(\Sigma;\mathcal P_{\Sigma})} \tilde L.
$$
Now let $h_{\pm\infty}$ is a differential form on 
$R(M;\mathcal P_{M}) \times_{R(\Sigma;\mathcal P_{\Sigma})} \tilde L$
and $h_1,\dots,h_k$ be differential forms on 
$\tilde L \times_{R(\Sigma;\mathcal P_{\Sigma})} \tilde L$.
The right filtered $A_{\infty}$ module structure is defined 
roughly speaking by
\begin{equation}\label{form67}
\aligned
&\langle 
\frak n_k([h_{-\infty}];[h_1],\dots,[h_k]),
[h_{+\infty}]\rangle \\
&=
\int_{\mathcal M((M;\mathcal P_M),L)}
{\rm ev}^*(h_1\times\dots h_k) \wedge 
{\rm ev}_{-\infty}^* h_{-\infty}
\wedge {\rm ev}_{+\infty}^* h_{+\infty}.
\endaligned
\end{equation}
(Note  we use integration in (\ref{form67}).
When we work on $\Z_2$ coefficient 
we actually need to use, say, singular homology rather than de Rham 
homology.)
\par
We can prove the relation (\ref{form63}) by studying the compactification 
of the moduli space $\mathcal M((M;\mathcal P_M),L)$
and its codimension one boundary.

\subsection{Proof of Gluing theorem}
\label{gluing}
In this subsection we sketch the proof of Theorem \ref{thm64}.
This proof is a `gauge theory analogue' of the 
proof by Lekili-Lipyanskiy \cite{LL} of a similar result in 
Lagrangian correspondence. (See Section \ref{lagcorr}.)
\begin{figure}[h]
\centering
\includegraphics[scale=0.3]{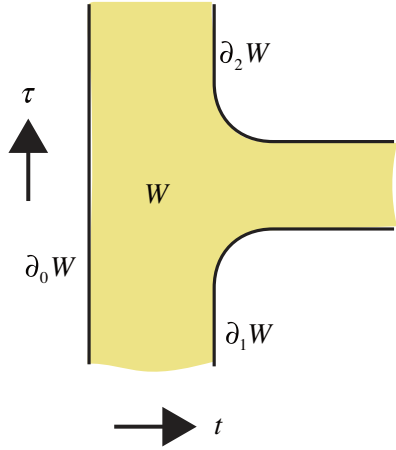}
\caption{The domain $W$}
\label{zu8}
\end{figure}
We consider a domain $W$ of $\C$ as in Figure \ref{zu8}.
It has three boundary components $\partial_0W$, $\partial_1W$, $\partial_2W$, 
which lie in the part $t=0$, $\tau<0$, $\tau>0$, respectively.
\par
We consider the direct product $\Sigma \times W$ with the direct product
metric.
We glue $M_1 \times \R_{\tau}$ with $W \times \Sigma$ by 
the diffeomorphism $\partial M_1 \times \R_{\tau} \cong \Sigma \times \partial_1W$.
We also glue $M_2 \times \R_{\tau}$ with $W \times \Sigma$ by 
the diffeomorphism $\partial M_2 \times \R_{\tau} \cong \Sigma \times \partial_2W$.
We then obtain a 4 manifold $X$ with boundary and ends.
$X$ has a boundary
$$
\partial X = \Sigma \times \partial_0W \cong \Sigma \times \R_{\tau}.
$$
$X$ has three ends.
\begin{enumerate}
\item[(End.1)]
$M_1 \times (-\infty,0]_{\tau}$. This lies in the part where $\tau \to -\infty$.
\item[(End.2)]
$M\times [0,\infty)_{t}$. Here $M$ is obtained by gluing $M_1$ and $M_2$ along 
$\Sigma$. This lies in the part where $t \to \infty$.
\item[(End.3)]
$M_1 \times [0,+\infty,0)_{\tau}$. This lies in the part where $\tau \to +\infty$.
\end{enumerate}
See Figure \ref{zu9}.

\begin{figure}[h]
\centering
\includegraphics[scale=0.3]{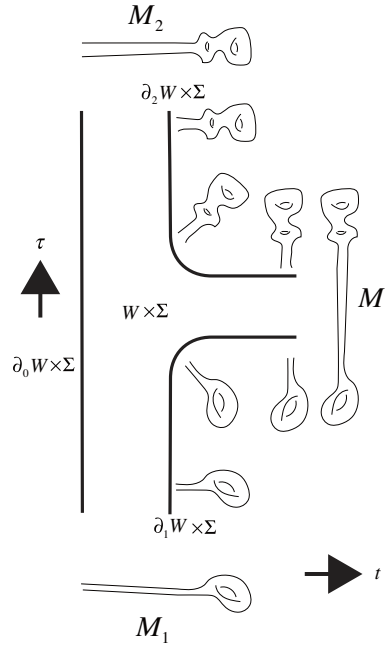}
\caption{The 4 manifold $X$}
\label{zu9}
\end{figure}

Note the $SO(3)$-bundles $\mathcal P_{M_1}$ and $\mathcal P_{M_2}$
induce an  $SO(3)$-bundle on $X$, which we denote by $\mathcal P_X$.
\par
We take 
$$
\alpha, \beta 
\in R(M_1;\mathcal P_{M_1}) \times_{R(\Sigma;\mathcal P_{\Sigma})} 
R(M_2;\mathcal P_{M_2}),
$$
and consider the connection $A$ of $\mathcal P_X$ with the 
following properties.
\begin{enumerate}
\item
$A$ is an ASD connection. Namely $F_A^+ = 0$.
\item
On boundary $(-\infty,0] \times \Sigma$
we use the Lagrangian submanifold $R(M_1;\mathcal P_{M_1})$ 
to set the boundary condition for $A$.
\item
On boundary $[0,+\infty,0) \times \Sigma$
we use the Lagrangian submanifold $R(M_2;\mathcal P_{M_2})$ 
to set the boundary condition for $A$.
\item
At $\{(0,0)\} \times \Sigma$ we require that the restriction of $A$ 
is $\alpha$. 
\item
At the end (End.1) we require that $A$ is asymptotic 
to an element of 
$$
R(M_1;\mathcal P_{M_1}) 
\subset R(M_1;\mathcal P_{M_1}) \times_{R(\Sigma;\mathcal P_{\Sigma})}
R(M_1;\mathcal P_{M_1}).
$$
\item
At the end (End.2) we require that $A$ is 
asymptotic to the flat connection $\beta$ of $(M,\mathcal P_M)$.
\item
At the end (End.3) we require that $A$ is asymptotic 
to an element of 
$$
R(M_2;\mathcal P_{M_2}) 
\subset R(M_2;\mathcal P_{M_2}) \times_{R(\Sigma;\mathcal P_{\Sigma})}
R(M_2;\mathcal P_{M_2}).
$$
\item
$$
\int_{X}\Vert F_A\Vert^2 {\rm Vol}_M d\tau
< \infty.
$$
\end{enumerate}
Note items (5) and (7) mean that $A$ is asymptotic to the 
cyclic element ${\bf 1}$ there.
See Figure \ref{zu10}.

\begin{figure}[h]
\centering
\includegraphics[scale=0.3]{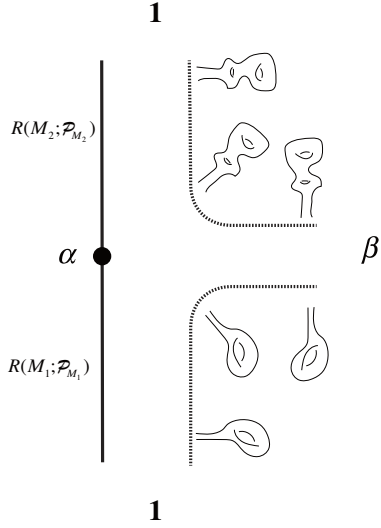}
\caption{Condition for connection $A$}
\label{zu10}
\end{figure}

\begin{defn}
We denote by
$\mathcal M((X,\mathcal P_X);\alpha,\beta)$ the moduli space
of gauge equivalence classes of the connections $A$ satisfying 
the conditions (1)-(8) above.
\end{defn}

We use this moduli space to define a map
$$
CF((R(M_1;\mathcal P_{M_1}),b_{M_1}),
(R(M_2;\mathcal P_{M_2}),b_{M_2}))
\to 
CF(M;\mathcal P_{M})
$$
by 
\begin{equation}\label{form68}
[\alpha]
\mapsto
\sum_{\beta} 
\#\mathcal M((X,\mathcal P_X);\alpha,\beta) [\beta].
\end{equation}
Here we use the component of the moduli space
$\mathcal M((X,\mathcal P_X);\alpha,\beta)$ with 
(virtual) dimension $0$.
\par
To show that (\ref{form68}) becomes a 
chain map,
we study the compactification of $\mathcal M((X,\mathcal P_X);\alpha,\beta)$
and show that its codimension one boundary is classified as follows.

\begin{enumerate}
\item[(bdry.1)]
The disk bubble at the boundary point $(\tau,0)$ with $\tau < 0$.
\item[(bdry.2)]
Disk bubble at the boundary point $(\tau,0)$ with $\tau > 0$.
\item[(bdry.3)]
Disk bubble at $(0,0)$.
\item[(bdry.4)]
Sliding ends as $\tau \to -\infty$.
\item[(bdry.5)]
Sliding ends as $\tau \to +\infty$.
\item[(bdry.6)]
Sliding ends as $t \to +\infty$.
\end{enumerate}

See Figure \ref{zu11}.

\begin{figure}[h]
\centering
\includegraphics[scale=0.5]{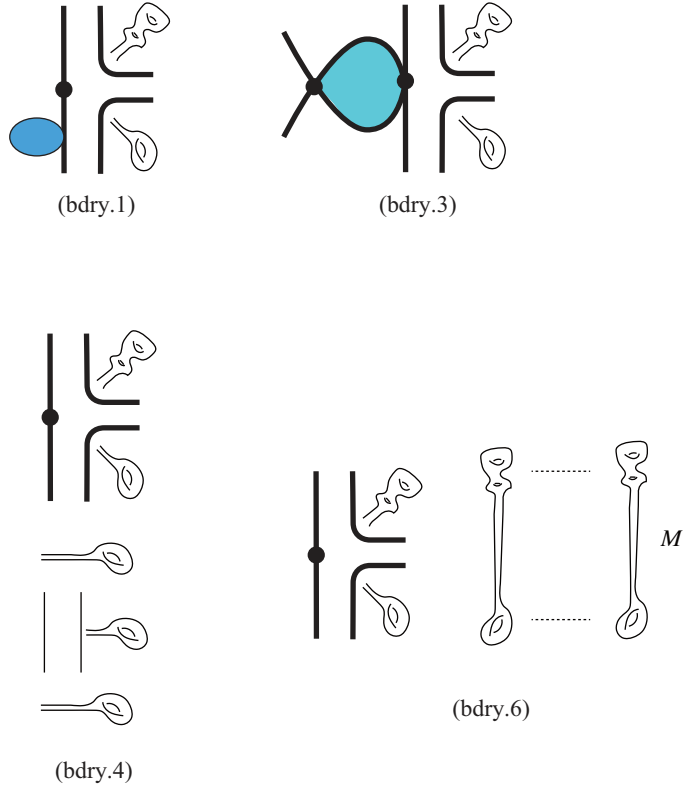}
\caption{The boundary of $\mathcal M((X,\mathcal P_X);\alpha,\beta)$}
\label{zu11}
\end{figure}

We can cancel the effect of the boundaries as in (bdry.1) and (bdry.2)
using bounding cochains $b_{M_1}$ and $b_{M_2}$.
(This part of the argument is the same as the Lagrangian Floer theory
\cite{fooobook}.)
The effect of boudaries as in (bdry.4) and (bdry.5) is zero 
because of the equality
$$
d^{b_{M_1}} {\bf 1} = d^{b_{M_2}} {\bf 1} =0.
$$
(This is the equality (\ref{form65}), which we required when we defined
$b_{M_1}$ and $b_{M_2}$.)
\par
Therefore the remaining boundary components are ones of 
(bdry.3) and (bdry.6).
\par
(bdry.3) is described by the sum of the product
\begin{equation}\label{form69}
{\mathcal M}(R(M_1;\mathcal P_{M_1}),R(M_2;\mathcal P_{M_2});
\alpha,\alpha')
\times \mathcal M((X,\mathcal P_X);\alpha',\beta)
\end{equation}
for various $\alpha'$.
Here the first factor is a special 
case of the moduli space
${\mathcal M}(L_1,L_2;
a,b)$, which we introduced in Subsection \ref{subsec:Laggeneral},
using (LF.1),(LF.2),(LF.3).
\par
(bdry.6) is described by the sum of the product
\begin{equation}\label{form610}
\mathcal M((X,\mathcal P_X);\alpha,\beta')
\times 
{\mathcal M}(M\times \R;\beta',\beta)
\end{equation}
for various $\beta'$.
Here the second factor is the moduli space 
introduced in Subsection \ref{subsec:Floer homology}
using condition (IF.1),(IF.2),(IF.3).
\par
In the simplest case where $b_{M_1} = b_{M_2} = 0$
the above argument implies that 
the sum of (\ref{form69}) and (\ref{form610})
is $0$ (in $\Z_2$ coefficient in our situation).
It implies that the map (\ref{form68}) is a chain map.
\par
In the general case we need additional term which is related 
to the correction by 
$b_{M_i}$.
It is described by the moduli space drawn in Figure \ref{zu12}.
(We omit the detail.)
\begin{figure}[h]
\centering
\includegraphics[scale=0.3]{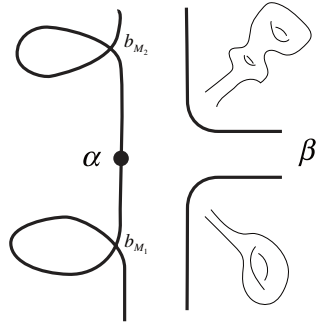}
\caption{The correction to the map (\ref{form68}).}
\label{zu12}
\end{figure}
We again obtain a chain map
$$
\mathscr F : CF((R(M_1;\mathcal P_{M_1}),b_{M_1}),
(R(M_2;\mathcal P_{M_2}),b_{M_2}))
\to 
CF(M;\mathcal P_{M}).
$$
To show that $\mathscr F$ induces isomorphism 
we use energy filtration.
The leading order term of the map $\mathscr F$
with respect to the energy filtration is the case of 
energy $0$. It consists of  connections $A$ 
which is flat. In that case $\alpha$ is necessary 
equal to $\beta$. Thus the leading order term 
of the map $\mathscr F$
is the identity map, by using the identification
$$
R(M_1;\mathcal P_{M_1})
\times_{R(\Sigma;\mathcal P_{\Sigma})}
R(M_2;\mathcal P_{M_2})
\cong 
R(M;\mathcal P_{M}).
$$
Therefore $\mathscr F$ induces the required isomorphism
$$
HF((R(M_1;\mathcal P_{M_1}),b_{M_1}),
(R(M_2;\mathcal P_{M_2}),b_{M_2}))
\cong
HF(M;\mathcal P_{M}).
$$
This is the outline of the proof of Theorem \ref{thm64}.

\section{Lagrangian correspondence and $A_{\infty}$ functors.}
\label{lagcorr}

The story we described in Section \ref{boundary}
has an analogue in the study of Lagrangian correspondence,
which we will outline in this section.
See \cite{efl} for detail.
In this section we work over the ground ring $\R$.
We need to take a spin or relative spin structure of a Lagrangian 
submanifold to use such ground ring.
Spin or relative spin structures are necessary to orient the moduli spaces 
of pseudoholomorphic disks.
(See \cite[Chapter 8]{fooobook2}.)
If $(X,\omega)$ is a symplectic manifold and $V$ is an oriented real 
vector bundle on it the $V$-relative spin structure of an 
oriented submanifold
$L \subset X$ is by definition the spin structure of the bundle
$TL \oplus V\vert_L$. 
Under this assumption the construction of Subsection \ref{immersed}
works over the ground ring $\R$ (or $\Q$).

\subsection{The main results}
\label{maincorr}

Let $(X_1,\omega_1)$ and $(X_2,\omega_2)$ be compact symplectic manifolds.

\begin{defn}
An {\it immersed Lagrangian correspondence} from $X_1$ to $X_2$ is an immersed 
Lagrangian submanifold of $(X_1 \times X_2, -\omega_1 \oplus \omega_2)$.
\par
Let $L_{12} = (\tilde L_{12},i_{L_{12}})$ 
be an immersed Lagrangian correspondence from $X_1$ to $X_2$
and $L_1= (\tilde L_{1},i_{L_{1}})$ be an immersed Langrangian submanifold
of $X_1$. If the fiber product 
$\tilde L_2 = \tilde L_1 \times_{X_1} \tilde L_{12}$ is transversal 
then, together with the composition
$\tilde L_1 \times_{X_1} \tilde L_{12} \to \tilde L_{12}
\to X_1 \times X_2 \to X_2$, the manifold $\tilde L_{12}$ defines an immersed Lagrangian 
submanifold of $X_{2}$. We call it the 
{\it geometric transformation} of $L_1$ by $L_{12}$ 
and write it as $L_1 \times_{X_2} L_{12}$.
\par
Let $(X_i,\omega_i)$, $(i=1,2,3)$ be symplectic manifolds.
Let $L_{12}$ be an immersed Lagrangian submanifold of 
$(X_1 \times X_2, -\omega_1 \oplus \omega_2)$
and $L_{23}$ a Lagrangian submanifold of $(X_2 \times X_3, -\omega_2 \oplus \omega_3)$.
We assume that the fiber product 
$$
\tilde L_{12} \times_{X_2} \tilde L_{23} 
$$
is transversal and write the fiber product as $\tilde L_{13}$.
Together with the obvious map $i_{L_{13}} : \tilde L_{13} \to X_1 \times X_3$,
the manifold 
$\tilde L_{13}$ defines an immersed Lagrangian submanifold $L_{13}$ of
$(X_1 \times X_3, -\omega_1 \oplus \omega_3)$.
We call $L_{13}$ the {\it geometric composition} of $L_{12}$ and $L_{23}$.
We write it as $L_{12} \times_{X_2} L_{23}$.
\end{defn}
See \cite{wei}.
\begin{defn}
The (immersed) {\it Weinstein category} is defined as follows.
Its object is a symplectic manifold $(X,\omega)$.
A morphism from $X_1$ to $X_2$ is an immersed Lagrangian correspondence 
from $X_1$ to $X_2$.
\par
The composition of morphisms is defined as their geometric composition.
\end{defn}
A slight issue is that, actually, we can define geometric composition only 
for a transversal pair.
However, for the purpose of most of the applications, we can go around this problem 
by considering only composable pair of morphisms.
In other words, Weinstein category is rather a `topological category' where morphisms can 
be composed only on certain dense open subset.
We can thus go around the problem by carefully stating 
various theorems in this subsection in such a way using 
only transversal pair for compositions.
Another possible way to proceed is to introduce certain equivalence 
relation between Lagrangian submanifolds such as 
Hamiltonian isotopy or Lagrangian cobordism so that we can compose 
Lagrangian correspondences after perturbing them 
in the equivalence classes.
\par\smallskip
In Subsection \ref{immersed} we start with a 
finite set of immersed Lagrangian submanifolds 
$\frak L$ of $(X,\omega)$ 
(that is, a clean collection) and obtained a filtered 
$A_{\infty}$ category, the set of  whose objects 
is $\frak L$. We denote it by $\frak{Fuk}(\frak L)$.
Roughly speaking we can take `all' immersed Lagrangian submanifolds
and define $\frak{Fuk}(X,\omega)$. An issue in doing so is perturbing 
all the Lagrangian submanifolds simultaneously to obtain some clean 
collection. We do not discuss this point. Using $\frak{Fuk}(\frak L)$
instead of  $\frak{Fuk}(X,\omega)$ is enough for the purpose of 
most of the applications.
To simplify the notation we pretend as if we defined 
the filtered $A_{\infty}$ category $\frak{Fuk}(X,\omega)$.
The actual result we prove is one which is restated by using $\frak{Fuk}(\frak L)$  instead.
\par
Note the filtered $A_{\infty}$ category $\frak{Fuk}(X,\omega)$ is in general 
curved. We denote by $\frak{Fuks}(X,\omega)$ the strict category 
associated to $\frak{Fuk}(X,\omega)$.
\par
In fact to take care of the problem of orientation and sign we need to 
use relative spin structure. We fix $V$ and consider a set of 
triples $(L,\sigma,b)$ where $L$ is a Lagrangian submanifold of $X$ 
and $\sigma$ is a $V$-relative spin structure and $b$ is a bounding cochain of 
$CF(L)$, that is the (curved) $A_{\infty}$ algebra obtained by using $(L,\sigma)$.
The strict filtered $A_{\infty}$ category whose object is such triple $(L,\sigma,b)$
is abbreviated by $\frak{Fuks}(X,\omega,V)$.

\begin{defn}\label{unobwein}
An {\it unobstructed immersed Weinstein category} is defined as follows.
\begin{enumerate}
\item
Its object is a triple $(X,\omega,V)$ where $(X,\omega)$ is a 
compact symplectic manifold and $V$ is a real oriented vector bundle on $X$.
\item
A morphism from $(X_1,\omega_1,V_1)$ to $(X_2,\omega_2,V_2)$ 
is a triple $(L_{12},\sigma_{12},b_{12})$ where 
\begin{enumerate}
\item
$L_{12}$ is an immersed Lagrangian submanifold of $(X_1 \times X_2, -\omega_1 \oplus \omega_2)$.
\item
$\sigma_{12}$ is a $\pi_1^*V_1 \oplus \pi_1^*TX_1 \oplus \pi_2^*V_2 $-
relative spin structure of $L_{12}$.
\item
$b_{12}$ is a bounding cochain of $CF(L_{12})$.
(Note the filtered $A_{\infty}$ algebra $CF(L_{12})$ is defined 
by using the relative spin structure in (b).)
\end{enumerate}
\item
See Theorem \ref{thm75} for the composition of the morphisms.
\end{enumerate}
\end{defn}
The main result of \cite{efl} is a construction of the 
(2-)functor from the unobstructed immersed Weinstein category 
to the (2-)category of all filtered $A_{\infty}$ categories.
We will state it as Theorems \ref{thm74}, \ref{thm75}, 
\ref{thm76} below.

\begin{thm}\label{thm74}{\rm(\cite{efl})}
Let $(L_{12},\sigma_{12},b_{12})$ be as in Definition \ref{unobwein}
(2).
\begin{enumerate} 
\item
Let $(L_1,\sigma_1,b_1)$ be an object of $\frak{Fuks}(X_1,\omega_1,V_1)$.
Then the geometric transformation $L_1 \times_{X_1} L_{12}$
has a canonical choice of $V_2$ relative spin structure $\sigma_2$ and 
a bounding cochain $b_2$.
\item
There exists a strict filtered $A_{\infty}$ functor 
$$
\frak W_{L_{12}}:
\frak{Fuks}(X_1,\omega_1,V_1) \to \frak{Fuks}(X_2,\omega_2,V_2)
$$
of which the map $(L_1,\sigma_1,b_1) \mapsto (L_1 \times_{X_1} L_{12},
V_2,b_2)$ in item (1) is the object part.
\item
There is a strict filtered $A_{\infty}$ bifunctor
$$
\aligned
\frak{Fuks}(X_1 \times X_2, -\omega_1 \oplus \omega_2,
\pi_1^*V_1 \oplus \pi_1^*TX_1 \oplus \pi_2^*V_2 )
&\times
\frak{Fuks}(X_1,\omega_1,V_1) \\
&\to \frak{Fuks}(X_2,\omega_2,V_2)
\endaligned
$$
which induces $\frak W_{L_{12}}$ when we fix an object of 
the first factor
$\frak{Fuks}(X_1 \times X_2, -\omega_1 \oplus \omega_2,
\pi_1^*V_1 \oplus \pi_1^*TX_1 \oplus \pi_2^*V_2 )$.
\end{enumerate} 
\end{thm}
The notions of filtered $A_{\infty}$ functor and bifunctor 
(and its strictness) are explained in the next 
subsection.
\begin{rem}
If we assume all the Lagrangian submanifolds involved 
(including those appearing as fiber products 
among the Lagrangian submanifolds) are 
embedded, monotone and have minimal Maslov number $>2$, 
Theorem \ref{thm74} follows from the earlier results 
by Wehrheim-Woodward \cite{WW,WW2}
and 
Ma'u-Wehrheim-Woodwards \cite{MWW}.
The same remark applies to Theorems \ref{thm75} and \ref{thm76}.
\end{rem}
Note Theorem \ref{thm74} (3) implies that there exists a strict 
filtered $A_{\infty}$ functor
\begin{equation}\label{form71}
\aligned
&\frak{Fuks}(X_1 \times X_2, -\omega_1 \oplus \omega_2,
\pi_1^*V_1 \oplus \pi_1^*TX_1 \oplus \pi_2^*V_2 ) \\
&\to
\frak{Funcs}
(\frak{Fuks}(X_1,\omega_1,V_1),
\frak{Fuks}(X_2,\omega_2,V_2)).
\endaligned
\end{equation}
Here $\frak{Funcs}(\mathscr C_1,\mathscr C_2)$ is the strict filtered $A_{\infty}$ is 
category whose object is a strict filtered $A_{\infty}$ functor 
$: \mathscr C_1 \to \mathscr C_2$.
\begin{thm}\label{thm75}{\rm(\cite{efl})}
Let $(L_{12},\sigma_{12},b_{12})$ (resp. $(L_{23},\sigma_{23},b_{23})$)
be morphisms from $(X_1,\omega_1,V_1)$ to $(X_2,\omega_2,V_2)$
(resp. from $(X_2,\omega_2,V_2)$ to $(X_3,\omega_3,V_3)$).
\begin{enumerate}
\item
Let $L_{13} = L_{12} \times_{X_2} L_{23}$ be the geometric composition.
Then we can define a $\pi_2^*V_2 \oplus \pi_2^*TX_2 \oplus \pi_3^*V_3$-
relative spin structure $\sigma_{13}$ on it and 
a bounding cochain $b_{13}$ of $CF(L_{13})$.
(In particular $L_{13}$ is unobstructed.) 
\item
Let $\frak W_{L_{12}}$, $\frak W_{L_{23}}$ and $\frak W_{L_{13}}$ 
be strict filtered $A_{\infty}$ functors associated to 
$(L_{12},\sigma_{12},b_{12})$, $(L_{23},\sigma_{23},b_{23})$
and $(L_{13},\sigma_{13},b_{13})$, respectively, by 
Theorem \ref{thm74} (2).
Then 
\begin{equation}\label{form72}
\frak W_{L_{23}} \circ \frak W_{L_{12}} \sim \frak W_{L_{13}}.
\end{equation}
Here the left hand side is the composition of strict filtered $A_{\infty}$
functors and $\sim$ is the homotopy equivalence of two strict  filtered $A_{\infty}$ 
functors.
\item
The next diagram commutes up to homotopy equivalence of 
strict filtered $A_{\infty}$ bifunctors.
$$
\begin{CD}
\frak{Fuks}(X_1 \times X_2)
\times \frak{Fuks}(X_2 \times X_3) @ >>>
\frak{Fuks}(X_1 \times X_3) \\
@ VVV @ VVV\\
\displaystyle \frak{Funcs}(\frak{Fuks}(X_1),\frak{Fuks}(X_2)) 
\atop \displaystyle \times \frak{Funcs}(\frak{Fuks}(X_2),\frak{Fuks}(X_3)) @ > >> \frak{Funcs}(\frak{Fuks}(X_1),\frak{Fuks}(X_3))
\end{CD}
$$
Here the vertical arrows are (\ref{form71}).
(We omit the bundle $V_i$ in the notation.)
The first horizontal line is a strict filtered $A_{\infty}$ bifunctor 
whose object part sends $((L_{12},\sigma_{12},b_{12}),
(L_{23},\sigma_{23},b_{23}))$ to $(L_{13},\sigma_{13},b_{13})$
in item (1).
\par
The second horizontal line is a strict filtered $A_{\infty}$ bifunctor
whose object part is a composition of filtered $A_{\infty}$ 
functors.
\par
(Note the commutativity of this diagram in the object level 
is (\ref{form72}).)
\end{enumerate}
\end{thm}
\begin{thm}\label{thm76}{\rm(\cite{efl})}
The next diagram commutes up to  homotopy equivalence of 
strict filtered $A_{\infty}$ tri-functors.
$$
\begin{CD}
\displaystyle \frak{Fuks}(X_1 \times X_2)
\times \frak{Fuks}(X_2 \times X_{3})
\atop\displaystyle 
\!\!\!\!\!\!\!\!\!\!\!\!\!\!\!\!\!\!\!\!\!\!\!\!\!\!\!\!\!\!\!\!\!\!\!\!\!\!\!\!\!\!\times
\frak{Fuks}(X_3 \times X_{4}) @ >>>
\frak{Fuks}(X_1 \times X_{3})
\times\frak{Fuks}(X_3 \times X_{4}) \\
@ VVV @ VVV\\
\frak{Fuks}(X_1 \times X_{2})  
\times \frak{Fuks}(X_2 \times X_{4}) @ > >> \frak{Fuks}(X_1 \times X_{4}).
\end{CD}
$$
\par\smallskip
\noindent
where all the arrows are defined by composition functor in 
Theorem \ref{thm75}.
\end{thm}
Theorems \ref{thm74}, \ref{thm75}, 
\ref{thm76} provide the functorial picture of the 
construction of $A_{\infty}$ categories out of symplectic 
manifolds.
\begin{rem}
Theorems \ref{thm74}, \ref{thm75}, 
\ref{thm76} provide the functorial picture up to the level of 2-category.
Since the diagram in Theorem \ref{thm76} 
commutes only up to homotopy, we may continue and may have 
a compatibility as $\infty$-categories.
\end{rem}

\subsection{$A_{\infty}$ functor, $A_{\infty}$ bi-functor, 
and Yoneda's lemma}
\label{Yoneda}

The proof of Theorem \ref{them61} which we explained in 
Subsection \ref{exist} could be regarded as a proof using the idea of 
representable functors.
For the proof of Theorems \ref{thm74}, \ref{thm75}, 
\ref{thm76} we use a similar idea in more systematic way.
In this subsection we explain certain definitions and 
results
in the story of $A_{\infty}$ category needed for this purpose.
See \cite{fu4,Ke,Lef,Ly,sei} etc. for homological 
algebra of $A_{\infty}$ category.
The notion of $A_{\infty}$ bi-functor is discussed in more 
detail in \cite{efl}.
\par
Let $\mathscr C$ be an filtered $A_{\infty}$ category
and $c,c' \in \frak{OB}(\mathscr C)$.
We define
\begin{equation}
B_k \mathscr C[1](c,c')
= 
\bigoplus \bigotimes_{i=1}^k \mathscr C[1](c_{i-1},c_i)
\end{equation}
where direct sum is taken over all $c_0,\dots,c_k$ such that
$c_0 = c$ and $c_k =c'$,
and 
\begin{equation}
B \mathscr C[1](c,c') 
= \bigoplus_{k=0}^{\infty}B_k \mathscr C[1](c,c').
\end{equation}
$B \mathscr C[1](c,c')$ has a coalgebra structure 
$$
\Delta : B \mathscr C[1](c,c') 
\to \bigoplus_{c''} B \mathscr C[1](c,c'')  \,\hat\otimes B \, \mathscr C[1](c'',c') 
$$
defined by
$$
\Delta(x_1\otimes \dots \otimes x_k)
= 
\sum_{i=1}^{k-1} (x_1 \otimes \dots \otimes x_i)
\otimes (x_{i+1} \otimes \dots \otimes x_k).
$$
The operation $\frak m_k$ induces a unique coderivation 
$B \mathscr C[1](c,c') \to B \mathscr C[1](c,c')$ so that its  
$Hom(B_k\mathscr C[1](c,c'),B_1\mathscr C[1](c,c'))$ component is 
$\frak m_k$. We denote it by $\hat d_k$ 
and put 
$$
\hat d = \sum \hat d_k : B \mathscr C[1](c,c') \to B \mathscr C[1](c,c').
$$
The $A_{\infty}$ relation (\ref{formula25}) is equivalent to 
the equality $\hat d \circ \hat d = 0$.

\begin{defn}
Let $\mathscr C_1, \mathscr C_2$ be filtered $A_{\infty}$
categories.
A filtered $A_{\infty}$ functor $\mathscr F : \mathscr C_1 \to \mathscr C_2$
consists of objects:
\begin{enumerate}
\item
A map $\mathscr F_{ob} : \frak{OB}(\mathscr C_1) \to \frak{OB}(\mathscr C_2)$.
\item
A series of maps
$$
\mathscr F_{c,c'} : B \mathscr C_1[1](c,c') \to \mathscr C_2[1](\mathscr F_{ob}(c),
\mathscr F_{ob}(c')),
$$
which preserves filtrations.
\item
We require that the coalgebra homomorphism
$$
\widehat{\mathscr F}_{c,c'} : B \mathscr C_1[1](c,c') \to B \mathscr C_2
[1](\mathscr F_{ob}(c),
\mathscr F_{ob}(c'))
$$
induced by $\mathscr F_{c,c'}$ is a chain map with respect to the 
boundary operator $\hat d$.
\item
We require
$$
\widehat{\mathscr F}_{c,c'}(x_1,\dots,{\bf e},\dots,x_k) = 0
$$
except
$$
\widehat{\mathscr F}_{c,c}({\bf e}_c) = {\bf e}_{{\mathscr F}_{ob}(c)}.
$$
\end{enumerate}
A filtered $A_{\infty}$ functor is said to be {\it strict} if its restriction 
to $B_0 \mathscr C_1[1](c,c)$ is $0$.
\end{defn}
For a pair of  filtered $A_{\infty}$ categories
$\mathscr C_1, \mathscr C_2$ we can define a 
filtered $A_{\infty}$ category $\frak{Func}(\mathscr C_1, \mathscr C_2)$
whose object is a filtered $A_{\infty}$ functor $:\mathscr C_1 \to \mathscr C_2$.
We can define its strict version $\frak{Funcs}(\mathscr C_1, \mathscr C_2)$
in the same way.

\begin{defn}\label{defn78}
Let $\mathscr C_1, \dots, \mathscr C_n$ and $\mathscr C$ be 
filtered $A_{\infty}$ categories.
A filtered $A_{\infty}$ multi-functor 
$$
\mathscr F : 
\mathscr C_1 \times \dots \times  \mathscr C_n \to \mathscr C
$$
consists of the following objects.
\begin{enumerate}
\item
A map $\mathscr F_{ob} : \prod_{i=1}^n\frak{OB}(\mathscr C_i) \to \frak{OB}(\mathscr C)$.
\item
A series of $\Lambda_0$ module homomorphisms
$$
\mathscr F_{(c_1,\dots,c_n),(c'_1,\dots,c'_n)} : 
\bigotimes_{i=1}^n B \mathscr C_i[1](c_i,c_i') \to \mathscr C(\mathscr F_{ob}(c_1,\dots,c_n),
\mathscr F_{ob}(c'_1,\dots,c'_n)),
$$
which preserve filtrations.
\item
We require that the coalgebra homomorphism
$$
\hat{\mathscr F}_{(c_1,\dots,c_n),(c'_1,\dots,c'_n)} : 
\bigotimes_{i=1}^n B \mathscr C_i[1](c_i,c_i') \to 
B\mathscr C[1](\mathscr F_{ob}(c_1,\dots,c_n),
\mathscr F_{ob}(c'_1,\dots,c'_n))
$$
induced by ${\mathscr F}_{(c_1,\dots,c_n),(c'_1,\dots,c'_n)}$ is a chain map with respect to the 
boundary operator $\hat d$.
\item
$\hat{\mathscr F}_{(c_1,\dots,c_n),(c'_1,\dots,c'_n)}(\bf x)$ is zero 
if $\bf x$ contains a unit except 
$$
\hat{\mathscr F}_{(c_1,\dots,c_n),(c_1,\dots,c_n)}({\bf e}_{c_1},\dots,{\bf e}_{c_n})
= {\bf e}_{\mathscr F_{ob}(c_1,\dots,c_n)}.
$$
\end{enumerate}
\end{defn}
We use coalgebra structures of $B \mathscr C_i[1](c_i,c_i')$
to define a coalgebra structure on 
$\bigotimes_{i=1}^n B \mathscr C_i[1](c_i,c_i')$
in an obvious way.
\par
We can define strictness of multi-functor in the same way.
\begin{lem}\label{lem79}
The following two objects can be identified.
\begin{enumerate}
\item 
A filtered $A_{\infty}$ bifunctor 
$
\mathscr F : 
\mathscr C_1 \times \mathscr C_2 \to \mathscr C
$.
\item
A filtered $A_{\infty}$ functor 
$
\mathscr F : 
\mathscr C_1 \to \frak{Funk}(\mathscr C_2,\mathscr C)
$.
\end{enumerate}
\end{lem}
\begin{defn}
\begin{enumerate}
\item
Let $\mathscr C$ be a strict filtered $A_{\infty}$ category 
and $c,c' \in \frak{OB}(\mathscr C)$. We say $c$ and $c'$ are 
homotopy equivalent if there exists
$x \in \mathscr C(c,c')$ and $x' \in \mathscr C(c',c)$
such that $\frak m_1(x) = \frak m_1(x') = 0$, 
$\frak m_2(x,x') - {\bf e}_c \in {\rm Im}(\frak m_1)$, 
$\frak m_2(x',x) - {\bf e}_{c'} \in {\rm Im}(\frak m_1)$.
\item
Two strict filtered $A_{\infty}$ functors $\mathscr F_1,\mathscr F_2 :
\mathscr C_1 \to \mathscr C_2$
are said to be homotopy equivalent each other if 
they are homopoty equivalent in the functor category 
$\frak{Funcs}(\mathscr C_1,\mathscr C_2)$ in the sense of 
(1).
\item
A strict filtered $A_{\infty}$ functor $\mathscr F :
\mathscr C_1 \to \mathscr C_2$ is said to be a homotopy 
equivalence if there exists a strict filtered $A_{\infty}$ functor $\mathscr G :
\mathscr C_2 \to \mathscr C_1$ such that the compositions 
$\mathscr F\circ\mathscr G$ and $\mathscr G\circ\mathscr F$ are 
homotopy equivalent to the identity functor.
\end{enumerate}
\end{defn}
Basic results in the story of $A_{\infty}$ category is the 
following two theorems.

\begin{thm}\label{them710}{\rm (Whitehead theorem for $A_{\infty}$
functor)}
Let $\mathscr F :
\mathscr C_1 \to \mathscr C_2$ 
be a strict filtered $A_{\infty}$ functor between strict filtered $A_{\infty}$ categories. It is a homotopy 
equivalence if the following two conditions are satisfied.
\begin{enumerate}
\item 
For any $c' \in \frak{OB}(\mathscr C_2)$ there 
exists $c \in \frak{OB}(\mathscr C_1)$ such that 
$\mathscr F_{ob}(c)$ is homotopy equivalent to $c'$.
\item
For any $c,c' \in \frak{OB}(\mathscr C_1)$ the chain map
$$
\mathscr F_{c,c'} : \mathscr C_1(c,c') \to \mathscr C_2(\mathscr F_{ob}(c),\mathscr F_{ob}(c'))
$$
is a chain homotopy equivalence.
\end{enumerate} 
\end{thm}
We omit the proof. See \cite{fu4}.
\par
Let $\frak D \subset \frak{OB}(\mathscr C)$.
We define a filtered $A_{\infty}$ category 
whose object set is $\frak D$ and the module of morphisms and structure
operations are obvious restriction to those of $\mathscr C$.
We call such a filtered $A_{\infty}$ category a {\it full subcategory} 
of $\mathcal C$.
\par
We denote by $\mathscr{CH}$ an $A_{\infty}$ category 
whose object is a chain complex and whose module of 
morphisms between two chain complexes is the set of linear maps among them, which is a chain complex.
The boundary operator of this chain complex is the operator induced 
from $\frak m_1$ in an obvious way.
The operation $\frak m_2$ is the composition of linear maps 
(up to sign).
$\frak m_3$ and all higher $\frak m_k$ are all zero.
\begin{defn}
Let $\mathscr C$ be a filtered $A_{\infty}$ category.
We define its {\it opposite category} $\mathscr C^{\rm op}$
as follows.
\begin{enumerate}
\item
$\frak{OB}(\mathscr C^{\rm op}) = \frak{OB}(\mathscr C)$.
\item
$\mathscr C^{\rm op}(c,c') = \mathscr C(c',c)$.
\item
$$
\frak m^{\rm op}_k(x_1,\dots,x_k)
=
(-1)^* \frak m_k(x_k,\dots,x_1)
$$
where 
$* = 1 + \sum_{1 \le i < j \le k} (\deg x_i+1)(\deg x_j+1)$. Here  $\frak m^{\rm op}$ is the structure operation 
of $\mathscr C^{\rm op}$.
\end{enumerate}
\end{defn}
\begin{thm}\label{thm712}{\rm (Yoneda's lemma for $A_{\infty}$
categories)}
Let $\mathscr C$ be a strict filtered $A_{\infty}$ category.
\begin{enumerate}
\item
There exists a filtered $A_{\infty}$ functor
$
\frak{YON} : \mathscr C \to \frak{Funcs}(\mathscr C^{\rm op},\mathscr{CH}).
$
\item
Let $c \in \frak{OB}(\mathscr C) = \frak{OB}(\mathscr C^{\rm op})$.
Then  $\frak{YON}_{ob}(c) : \frak{OB}(\mathscr C) \to 
\frak{OB}(\mathscr{CH})$
is a strict filtered $A_{\infty}$ functor which is defined in the object level by 
$$
c' \mapsto \mathcal C(c,c').
$$
\item
Let $\frak{Rep}(\mathscr C^{\rm op},\mathscr{CH})$ be a full subcategory 
of $\frak{Funcs}(\mathscr C^{\rm op},\mathscr{CH})$, 
whose objects are elements of  $\frak{OB}(\frak{Funcs}(\mathscr C^{\rm op},\mathscr{CH}))$
which is homotopy equivalent 
to the image of $\frak{YON}_{ob} : \frak{OB}(\frak{\mathscr C}) \to 
\frak{OB}(\mathscr C^{\rm op},\mathscr{CH})$.
Then 
$$
\frak{YON} : \mathscr C \to \frak{Rep}(\mathscr C^{\rm op},\mathscr{CH}).
$$
is a homotopy equivalence.
\end{enumerate}
\end{thm}
We omit the proof. See \cite{fu4}.
We call $\frak{YON}$ the {\it Yoneda functor}.
We say an element $\frak{Rep}(\mathscr C^{\rm op},\mathscr{CH})$ 
a {\it representable} functor.

\subsection{K\"unneth tri-functor and representability}
\label{representability}

In this subsection we sketch an argument to prove 
Theorem \ref{thm74} using the algebraic framework 
of Subsection \ref{Yoneda}.
Let 
$(L_i,\sigma_i)$ be a pair of Lagrangian submanifold of 
$X_i$ and its $V_i$-relative spin structure
for $i=1,2$
and $(L_{12},\sigma_{12})$ a 
pair of Lagrangian submanifold of 
$(X_1 \times X_2, -\omega_1 \oplus \omega_2)$ and its $\pi_1^*V_1 \oplus \pi_1^*TX_1 \oplus \pi_2^*V_2 $-relative spin structure.
We consider curved filtered $A_{\infty}$ 
algebras $CF(L_i)$ and $CF(L_{12})$.
We assume appropriate transversality or clean-ness of intersection 
(or fiber product) among them.

\begin{prop}\label{prop713}
There exists a 
left $CF(L_1)$, $CF(L_{12})$ and right $CF(L_2)$ filtered $A_{\infty}$ tri-mudule 
$D$ such that as a $\Lambda_0$ module $D$ is
$$
H(\tilde L_1 \times_{X_1} \tilde L_{12} \times_{X_2} \tilde L_2;\Lambda_0)
$$
or its chain model.
\end{prop}
The  notion of tri-module is defined in a similar way as tri-functor
(Definition \ref{defn78}).
More explicitely it gives a series of operators
\begin{equation}\label{form75}
\frak n_{k_1,k_{12},k_2} :
CF(L_1)^{\otimes k_1} \otimes CF(L_{12})^{\otimes k_{12}}
\otimes D \otimes CF(L_{2})^{\otimes k_{2}}
\to 
D
\end{equation}
which satisfies a smilar relation as right module.
We sketch the proof of Proposition \ref{prop713} later in this 
subsection.

\begin{cor}\label{cor14}
If $b_1$ and $b_{12}$ are bounding cochains of $CF(L_1)$ and 
$CF(L_{12})$ respectively then 
$D$ in Proposition \ref{prop713} has a structure of right 
filtered  $A_{\infty}$ module over $CF(L_2)$.
\end{cor}
\begin{proof}
Using (\ref{form75}) we obtain
$$
\frak n_{k} :
D \otimes CF(L_{2})^{\otimes k}
\to 
D
$$
by
$$
\frak n_{k}(y;x_1,\dots,x_k)
=
\sum_{k_1,k_{12}=0}^{\infty}
\frak n_{k_1,k_{12},k}(b_1,\dots,b_1;b_{12},\dots,b_{12};y;x_1,\dots,x_k).
$$
It is easy to check (\ref{form63}).
\end{proof}
Now we consider the case when $L_{2}$ is the geometric 
transformation $L_{1} \times_{X_1} L_{12}$. Then 
$$
D = H(\tilde L_1 \times_{X_1} \tilde L_{12} \times_{X_2} \tilde L_2;\Lambda_0)
= H(\tilde L_{2} \times_{X_2} \tilde L_{2};\Lambda_0). 
$$
The fundamental class of $\tilde L_2$ is an element of $D$, 
which we write ${\bf 1}$.
\begin{lem}\label{lem715}
In the situation of Corollary \ref{cor14} we assume 
that $L_{2}$ is the geometric 
transformation $L_{1} \times_{X_1} L_{12}$.
\par
Then ${\bf 1} \in D$ is a cyclic element of right filtered 
$A_{\infty}$ module $D$.
\end{lem}
Theorem \ref{thm74} (1) is a consequence of 
Corollary \ref{cor14}, Lemma \ref{lem715} and 
Proposition \ref{prop611}.
\par
To prove Theorem \ref{thm74} (2)(3) we enhance 
Proposition \ref{prop713} as follows.
\begin{prop}\label{prop716}
There exists a filtered $A_{\infty}$ tri-functor
$$
\aligned
&\frak{Fuks}(X_1 \times X_2, -\omega_1 \oplus \omega_2,
\pi_1^*V_1 \oplus \pi_1^*TX_1 \oplus \pi_2^*V_2 )
\\
&\qquad\qquad\qquad\times
\frak{Fuks}(X_1,\omega_1,V_1) \times \frak{Fuks}(X_2,\omega_2,V_2)^{\rm op}
\to \mathscr{CH}
\endaligned
$$
such that the chain complex associated to 
$L_{12}$, $L_1$, $L_2$ by this tri-functor is $D$ in 
Proposition \ref{prop713}.
\end{prop}
By Proposition \ref{prop716}, Lemma \ref{lem79} induces a bifunctor
\begin{equation}\label{form706}
\aligned
&\frak{Fuks}(X_1 \times X_2, -\omega_1 \oplus \omega_2,
\pi_1^*V_1 \oplus \pi_1^*TX_1 \oplus \pi_2^*V_2 )
\\
&
\times
\frak{Fuks}(X_1,\omega_1,V_1)\to
\frak{Funcs}(\frak{Fuks}(X_2,\omega_2,V_2)^{\rm op},\mathscr{CH}). 
\endaligned
\end{equation}
\begin{prop}\label{lem717}
Let $(L_1,\sigma_1,b_1)$ and $(L_{12},\sigma_{12},b_{12})$
be objects of $\frak{Fuks}(X_1,\omega_1,V_1)$ and 
$\frak{Fuks}(X_1 \times X_2, -\omega_1 \oplus \omega_2,
\pi_1^*V_1 \oplus \pi_1^*TX_1 \oplus \pi_2^*V_2 )$,
respectively.
\par
Then the strict 
filtered $A_{\infty}$ functor $:\frak{Fuks}(X_2,\omega_2,V_2)^{\rm op} \to \mathscr{CH}$
obtained by applying (\ref{form706}) to them
is represented by the object $(L_2,\sigma_2,b_2)$.
\par
Here $L_2$ is the geometric composition and 
$b_2$ is obtained by Theorem \ref{thm74} (1).
\end{prop}
The proof is similar to the discussion in the next section 
using a diagram similar to the $Y$ diagram.
See \cite{efl}.
\par
By Proposition \ref{lem717} we obtain a 
filtered $A_{\infty}$ bi-functor
$$
\aligned
&\frak{Fuks}(X_1 \times X_2, -\omega_1 \oplus \omega_2,
\pi_1^*V_1 \oplus \pi_1^*TX_1 \oplus \pi_2^*V_2 )
\\
&
\times
\frak{Fuks}(X_1,\omega_1,V_1)\to
\frak{Rep}(\frak{Fuks}(X_2,\omega_2,V_2)^{\rm op},\mathscr{CH}). 
\endaligned
$$
Therefore we use Theorem \ref{thm712} (3) and 
compose homotopy inverse to the Yoneda functor 
to obtain desired filtered $A_{\infty}$
functor in Theorem \ref{thm74}.
\par
We finally sketch a proof of Proposition \ref{prop713}.
We use the moduli space of objects drawn in the next Figure
\ref{Figure51}. 

\begin{figure}[h]
\centering
\includegraphics[scale=0.35]{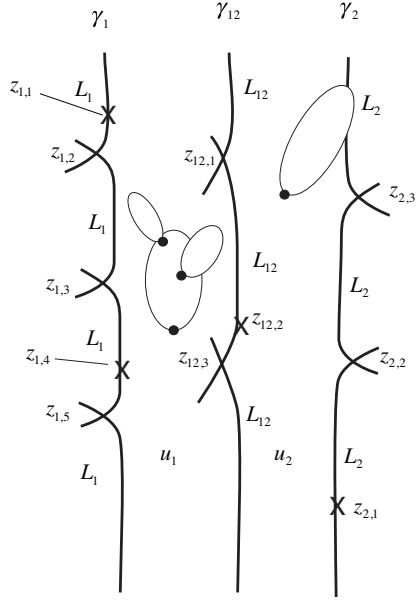}
\caption{Moduli space of simple quilt}
\label{Figure51}
\end{figure}

Here the source curve $\Sigma$ is the domain 
$\R_{\tau} \times [-1,1]_t$ plus possibly some sphere
bubbles. We divide $\Sigma$ into two parts.
The first one $\Sigma_1$ is the union of $\R_{\tau} \times [-1,0]_t$ 
and sphere bubbles 
rooted on it and the second one $\Sigma_2$ is the union of $\R_{\tau} \times [0,1]_t$ 
and sphere bubbles 
rooted on it.
(We require the sphere bubbles are not rooted on the part $t=-1,0,1$.)
The map is a combination of $u_1 : \Sigma_1 \to X_1$ and 
$u_2 : \Sigma_2 \to X_2$.
We include three kinds of marked points
$z_{1,1},\dots,z_{1,k_1} \in \R_{\tau} \times \{-1\}$, 
$z_{12,1},\dots,z_{12,k_{12}} \in \R_{\tau} \times \{0\}$, 
$z_{2,1},\dots,z_{2,k_2} \in \R_{\tau} \times \{1\}$.
\par
We require the following boundary conditions:
\begin{enumerate}
\item 
$u_1(\tau,-1) \in L_1$. 
\item
$u_2(\tau,+1) \in L_2$.
\item
$(u_{1}(\tau,0),u_{2}(\tau,0)) \in L_{12}$.
\item
$$
\lim_{\tau\to-\infty} (u_1(\tau,t_1),u_2(\tau,t_2))
= a \in \tilde L_1 \times_{X_1} \tilde L_{12} \times_{X_2}
\tilde L_2.
$$
\item
$$
\lim_{\tau\to+\infty} (u_1(\tau,t_1),u_2(\tau,t_2))
= b \in \tilde L_1 \times_{X_1} \tilde L_{12} \times_{X_2}
\tilde L_2.
$$
\end{enumerate}
Moreover we assume
\begin{enumerate}
\item[(6)]
$$
\int_{\Sigma_1} u_1^*\omega_1 + \int_{\Sigma_2} u_2^*\omega_2
= E  < \infty.
$$
\end{enumerate}
We denote the moduli space of such objects by 
$\mathcal M_{k_1,k_{12},k_2}(L_1,L_{12},L_2;a,b;E)$.
It comes with evaluation maps:
$$
{\rm ev} = ({\rm ev}_1,{\rm ev}_{12},{\rm ev}_2) : 
\mathcal M_{k_1,k_{12},k_2}(L_1,L_{12},L_2;a,b;E)
\to \hat L_1^{k_1} \times \hat L_{12}^{k_{12}}
\times \hat L_2^{k_2},
$$
where 
$$
\hat L_1 = \tilde L_1 \times_{X_1} \tilde L_1.
$$
$\hat L_2$ and $\hat L_{12}$ are defined in the same way.
\par
Let $h_{i,j}$ $(j=1,\dots,k_i)$ be differential forms on $\hat L_i$ and 
$h_{12,j}$ $(j=1,\dots,k_{12})$ differential forms on $\hat L_{12}$.
Then we define the structure operations (\ref{form75}) of the tri-module $D$ by
$$
\aligned
&\frak n_{k_1,k_{12},k_2}(\vec h_1;\vec h_{12};[a];\vec h_2) \\
&=
\sum_{E,b} T^E [b] \int_{\mathcal M_{k_1,k_{12},k_2}(L_1,L_{12},L_2;a,b;E)}{\rm ev}^*(\vec h_1\wedge \vec h_{12} \wedge \vec h_2).
\endaligned
$$
Here $\vec h_1 = (h_{1,1},\dots,h_{1,k_1}) \in CF(L_1)^{\otimes k_1}$.
The notations $\vec h_{12}$ and $\vec h_2$ are defined in a similar way.
\par
We can show that it satisfies the required relation by studying the 
boundary of our moduli space $\mathcal M_{k_1,k_{12},k_2}(L_1,L_{12},L_2;a,b;E)$ 
and using Stokes' theorem.
\par\smallskip
We remark that our moduli space  $\mathcal M_{k_1,k_{12},k_2}(L_1,L_{12},L_2;a,b;E)$
is similar to one 
we use to define Floer homology group
$
HF(L_1\times L_2,L_{12})
$
in the product space $X_1 \times -X_2$.
For example in case $L_1,L_2,L_{12}$ are embedded and monotone with minimal 
Maslov number $>2 $ we may use the case $k_1 = k_{12} = k_2 = 0$ 
to obtain
$$
\frak n_{0,0,0} : D \to D
$$
and $D$ is a free $\Lambda_{0}$ module with basis
$$
L_1 \times_{X_1} L_{12} \times_{X_2} L_2
=
(L_1 \times L_2) \cap L_{12}.
$$
So $D$ is also the underlying vector space of the chain complex 
calculating Floer homology $HF(L_1\times L_2,L_{12})$.
Moreover the operation $\frak n_{0,0,0}$ coincides with Floer's boundary operator.
\begin{rem}
To obtain appropriate Kuranishi structure 
we need to slightly change the way to compactify the 
bubble on the line $t=0$.
(See \cite[Section 12]{efl}.)
\par
In case $k_i$ or $k_{12}$ are nonzero
there is some technical difference 
between the moduli space $\mathcal M_{k_1,k_{12},k_2}(L_1,L_{12},L_2;a,b;E)$
and the moduli space we use to 
define $CF(L_1\times L_2)$-$CF(L_{12})$ bimodule structure 
on $D$. 
\end{rem}
Wehrheim-Woodward \cite{WW} studied the case of monotone 
Lagrangian submanifolds using the moduli space
$\mathcal M_{0,0,0}(L_1,L_{12},L_2;a,b;E)$.
They also generalize this moduli space to the case 
where the domain is divided 
into several (not necessary two) domains which are sent to 
various symplectic manifolds by a pseudoholomorphic curve.
They call such objects pseudoholomorphic quilt.
Here we use only the simplest case of  pseudoholomorphic quilt.

\subsection{Y diagram and compatibility of compositions}
\label{compatibilitycomp}

In this subsection we give a brief explanation of
the proof of Theorem \ref{thm75}.
\par
The proof of Theorem \ref{thm75} (1) is similar to
one of  Theorem \ref{thm74} (1).
(In fact Theorem \ref{thm74} (1) can be regarded as a 
special case of Theorem \ref{thm75} (1) where $L_3$ is a point.)
We construct a tri-module $D$ over 
$CF(L_{12})$, $CF(L_{23})$, $CF(L_{31})$ and use it in the same 
way as the last subsection.
The moduli space we use for this purpose is obtained 
by replacing Figure
\ref{Figure51} by the next Figure \ref{Figure82}.

\begin{figure}[h]
\centering
\includegraphics[scale=0.3]{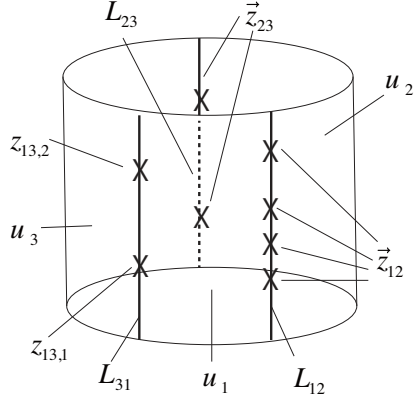}
\caption{Composition of unobstructed correspondences}
\label{Figure82}
\end{figure}

Here the source is a circular cylinder divided into three parts.
The map is a combination of three maps $u_1,u_2,u_3$ which 
sends those three parts to $X_1$, $X_2$ and $X_3$ respectively.
We use $L_{12}$, $L_{23}$, $L_{31}$ to set the boundary conditions 
at the lines where two of those three subdomains intersect.
The underlying vector space of $D$ is the cohomology group
of the triple fiber product
$$
\left\{(x,y,z) \in \tilde L_{12} \times \tilde L_{23} \times \tilde L_{31} 
\,\, \left\vert\,\,
\aligned
&\pi_1(i_{L_{12}}(x)) = \pi_1(i_{L_{31}}(z)), \\
&\pi_2(i_{L_{12}}(x)) = \pi_2(i_{L_{23}}(z)), \\
&\pi_3(i_{L_{13}}(x)) = \pi_3(i_{L_{23}}(z)).
\endaligned
\right\}\right.
$$
Here $\pi_i : X_i \times X_j \to X_i$ and $\pi_j  : X_i \times X_j \to X_j$
are projections.
\par
We put elements of this triple fiber product at the part $\tau \to \pm \infty$ 
and use it as the 
asymptotic boundary condition.
\par
The proof of Theorem \ref{thm75} (2)(3) is based on the
moduli space introduced by 
Lekili-Lipyanskiy \cite{LL} which is drawn in the next Figure
\ref{Figure93}.

\begin{figure}[h]
\centering
\includegraphics[scale=0.3]{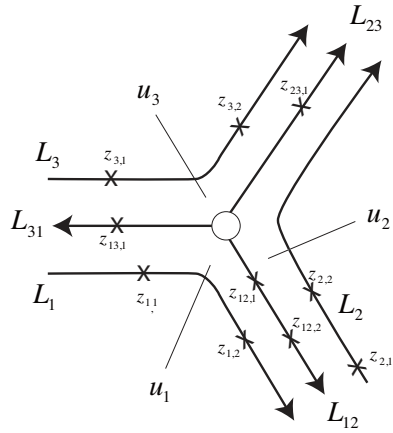}
\caption{Lekili-Lipyanskiy's Y-diagram}
\label{Figure93}
\end{figure}

Here the source is a domain in $\C$ which is divided into three parts.
The maps $u_1,u_2,u_3$ send each of those three parts to $X_1$, $X_2$, $X_3$,
respectively
and are pseudoholomorphic.
Note in our situation of Theorem \ref{thm75} (2)(3)
we are given 6 Lagrangian submanifolds $L_{12}$, $L_{23}$, $L_{31}$,
$L_1$, $L_2$, $L_3$, where $L_i \subset X_i$ and 
$L_{ij} \subset X_i \times X_j$.
We use $L_{ij}$ to set the boundary condition at the curves where
two subdomains intersect each other. We use $L_i$ to set the boundary 
condition at the boundary of our domain.
We have six curves and so we put six kinds of marked points.
The evaluation maps go to the products of $\hat L_i$'s and 
$\hat L_{ij}$'s.
\par
Note our domain has 4 ends. Three of them (left, upper right and lower right) 
are similar to the ends appearing in Figure \ref{Figure51}
and the fourth one which is a neighborhood of the white circle 
in the middle of the domain is similar to the end 
appearing in Figure \ref{Figure93}.
\par 
Thus our moduli space defines a map 
which interpolates tensor products of $CF(L_i)$, $CF(L_{ij})$ 
and tri-modules which we used to prove Theorem \ref{thm75} (1) and 
Theorem \ref{thm74} (1).
\par
The $\Lambda_0$ linear maps we thus obtain looks rather cumbersome.
However when we see them carefully we find that it is exactly 
the maps we need to show the homotopy commutativity of the 
diagrams appearing in Theorem \ref{thm75} (2)(3).
See \cite{efl} for detail.
\par
We finally mention that Theorem \ref{thm76} is proved by 
using the objects drawn in the next Figure \ref{Figure11-1}.

\begin{figure}[h]
\centering
\includegraphics[scale=0.4]{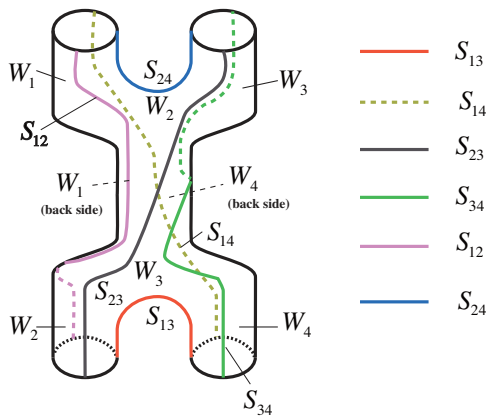}
\caption{Associativity of compositions}
\label{Figure11-1}
\end{figure}

\section{Categorification of Donaldson-Floer theory}
\label{DFcategori}

We can enhance the construction of Section \ref{boundary} to the topological 
filed theory style results and clarify its relation to 
Lagrangian correspondence.

For oriented two manifold $\Sigma$ we always consider 
an $SO(3)$ bundle $E_{\Sigma}$ on it such that 
$w_2(E_{\Sigma})$ is the fundamental class.
Let $R(\Sigma,E_{\Sigma})$
be the moduli space of the
gauge equivalence classes of the flat connections 
of $E_{\Sigma}$. It is a symplectic manifolds 
and is monotone with minimal Chern number 2.
\par
In this subsection we work over $\Z_2$ coefficient.
\begin{defn}
We consider the strict filtered $A_{\infty}$ category
$\frak{Fuk}(\Sigma)$
as follows.
\begin{enumerate}
\item
The object of $\frak{Fuk}(\Sigma)$ is 
a pair $(L,b)$. Here $L$ is an immersed Lagrangian 
submanifold of $R(\Sigma,E_{\Sigma})$, 
which is monotone in the weak sense (Definition \ref{defn4111}) 
and has minimal Maslov number divisible by $4$.
$b$ is its bounding cochain which is supported in the 
switching components.
(Here we consider the 
$\Lambda_0^{\Z_2}$ filtered $A_{\infty}$ algebra
associated to $L$.)
\item
The module of morphisms is $CF((L_1,b_1),(L_2,b_2))$
which is ($\Z_2$ version) of the chain complex introduced 
in Subsection \ref{subsec:Laggeneral}.
\item
The structure operations $\frak m_k$ is 
defined as in Subsection \ref{immersed}.
\end{enumerate}
We remark that we can easily prove the following which is expected by 
topological field theory.
\begin{equation}
\frak{Fuk}(\Sigma \sqcup \Sigma')
=
\frak{Fuk}(\Sigma)
\otimes
\frak{Fuk}(\Sigma').
\end{equation}
\begin{equation}
\frak{Fuk}(-\Sigma)
=
\frak{Fuk}(\Sigma)^{\rm op}.
\end{equation}
\end{defn}

Let $M$  be a 3 manifold with boundary
$\Sigma = \partial M$. We consider an  
$SO(3)$ bundle $E_M$ on $M$ such that 
the restriction of $E_M$ to 
$\Sigma = \partial M$ is $E_{\Sigma}$.
\footnote{When $M$ has a connected component which does 
not intersect with boundary, we require that the $SO(3)$ bundle 
$E_M$ is nontrivial on such a component.}
We assume that $\partial M$ is 
divided into $\partial_- M \sqcup \partial_+ M$
such that a neighborhood of $\partial_- M$ 
(resp. a neighborhood of $\partial_+ M$)
in $M$ is identified 
with $\partial_- M \times [-\infty,0)$
(resp. $\partial_+ M \times [0,+\infty)$)
as oriented manifold.
\par
By Theorem \ref{them61}, the 
(appropriately perturbed) moduli space of 
flat connections 
$R(M;E_{M})$ is unobstructed.
Namley we have a bounding cochain 
$b_M$ of the $\Lambda_{0}^{\Z_2}$ linear filtered 
$A_{\infty}$ algebra associated to $R(M;E_{M})
\subset R(\Sigma_-;E_{\Sigma_-})
\times -R(\Sigma_+;E_{\Sigma_+})$.

\begin{defn}\label{defn82}
Suppose $\Sigma_-,\Sigma_+ \ne \emptyset$.
We define the strict filtered $A_{\infty}$
functor
\begin{equation}
\mathcal{HF}_{(M,E_M)} : \frak{Fuk}(\Sigma_-)
\to \frak{Fuk}(\Sigma_+)
\end{equation}
as the strict filtered $A_{\infty}$
functor associated to the unobstructed immersed
Lagrangian correspondence $(R(M;E_{M}),b_M)$
by Theorem \ref{thm74}.
\end{defn}
We can prove the following compatibility 
result.
We consider $(M_i,E_{M_i})$ ($i=1,2$) as above.
We suppose $\partial_+M_1 \cong \partial_-M_2$.
We glue $M_1$ and $M_2$ along 
$\partial_+M_1 \cong \partial_-M_2$
to obtain $M_3 = M_1 \# M_3$ 
and an $SO(3)$ bundle $E_{M_3}$ on it.
Note 
$$
\partial_- M_3 = \partial_- M_1,
\qquad
\partial_+ M_3 = \partial_+ M_2.
$$
\begin{thm}\label{tjhem83}
The filtered $A_{\infty}$ functor 
$\mathcal{HF}_{(M_3,E_{M_3})} : \frak{Fuk}(\partial_- M_1)
\to \frak{Fuk}(\partial_+ M_2)$
associated to $(M_3,E_{M_3})$ by 
Definition \ref{defn82} is homotopy equivalent 
to the composition 
$$
\mathcal{HF}_{(M_2,E_{M_2})} \circ \mathcal{HF}_{(M_1,E_{M_1})}.
$$
Here $\mathcal{HF}_{(M_2,E_{M_2})}$, 
$\mathcal{HF}_{(M_1,E_{M_1})}$ are 
filtered $A_{\infty}$ functors associated to 
$(M_2,E_{M_2})$ and $(M_1,E_{M_1})$ by 
Definition \ref{defn82} and 
their composition is defined by Theorem \ref{thm75} (1).
\end{thm}
\begin{defn}\label{defn8484}
Suppose $\Sigma_- = \emptyset,\Sigma_+ \ne \emptyset$.
Then we define $\mathcal{HF}_{(M,E_M)}$
as the object $(R(M,E_M),b_M)$ of 
$\frak{Fuk}(\Sigma_+)$.
\par
Suppose $\Sigma_- \ne \empty,\Sigma_+ = \emptyset$.
Then we define $\mathcal{HF}_{(M,E_M)}$
as the filtered $A_{\infty}$ functor
$: \frak{Fuk}(\Sigma_-) \to \mathcal{CH}$
which is represented by the object
$(R(-M,E_{-M}),b_{-M})$ of $\frak{Fuk}(\Sigma_-)^{\rm op}$.
\par
Suppose $\Sigma_- = \Sigma_+ = \emptyset$.
Then we define $\mathcal{HF}_{(M,E_M)}$ 
as its Floer homology group as in Definition \ref{defn16}.
\end{defn}
We can generalize Theorem \ref{tjhem83}
appropriately including the situation of Definition \ref{defn8484}.
Since this generalization is straightforward 
we omit it.
\par
The proof of Theorem \ref{tjhem83} uses the following 
Figure \ref{Figure62}.
$X$ in the figure is a 4 manifold. $X$ has 3 ends and 3 boundary 
components. The ends are identified with
\begin{equation}\label{form8383}
M_1 \times (-\infty,0],
\qquad M_2 \times (-\infty,0].
\qquad
M_3 \times [0,+\infty,0).
\end{equation}
The boundary (which are drawn by dotted lines) are 
identifies with
$$
\partial_- M_1 \times \R
= \partial_- M_2 \times \R,
\quad
\partial_+ M_1 \times \R
= 
\partial_- M_2 \times \R
\quad
\partial_+ M_2 \times \R
= 
\partial_- M_3 \times \R.
$$
The free domains $\Omega_1$, $\Omega_2$, $\Omega_3$ 
of $\C$ 
are attached to each of such boundary components.
We consider an Anti-Self-Dual connection $A$ on $X$ 
and holomorphic maps 
$u_1 : \Omega_1 \to R(\partial_-M_1)$, 
$u_2 : \Omega_2 \to R(\partial_-M_2)$, 
$u_3 : \Omega_3 \to R(\partial_+M_3)$.
\par
Along three dotted lines we require appropriate 
matching condition similar to those in 
\cite{fu3,Li}.
We also require $u_1$, $u_2$, $u_3$ satisfy 
appropriate boundary condition 
on $\partial\Omega_i \setminus \text{\rm dotted lines}$,
formulated by using Lagrangian submanifolds 
$R(M_1;E_{M_1})$, $R(M_2;E_{M_2})$, $R(M_3;E_{M_4})$, 
respectively.
\par
We consider the moduli space of the such triples $(A,u_1,u_2,u_3)$.
(We also include boundary marked points on the 
$\partial\Omega_i \setminus \text{\rm dotted lines}$ and 
require certain asymptotic boundary conditions on the 
three ends.)
\par
We observe the sliding ends of this moduli space 
corresponding to the three ends in (\ref{form8383}) 
coincide with the moduli spaces we use 
to obtain bounding cochain 
$b_{M_1}$, $b_{M_2}$, $b_{M_3}$,  respectively.
\par
Using the moduli space of the triples $(A,u_1,u_2,u_3)$ (plus marked 
points),
we can show that $b_{M_1}$, $b_{M_2}$, $b_{M_3}$ satisfies 
certain equalities which is 
the one we need to prove Theorem \ref{tjhem83}.
\par
In this article we restrict ourselves to the case 
of $SO(3)$ bundles $E$ with nontrivial $w_2(E)$.
The research to include the case when $E$ is a 
trivial bundle is now in progress.
See \cite{DF}.

\begin{figure}[h]
\centering
\includegraphics[scale=0.4]{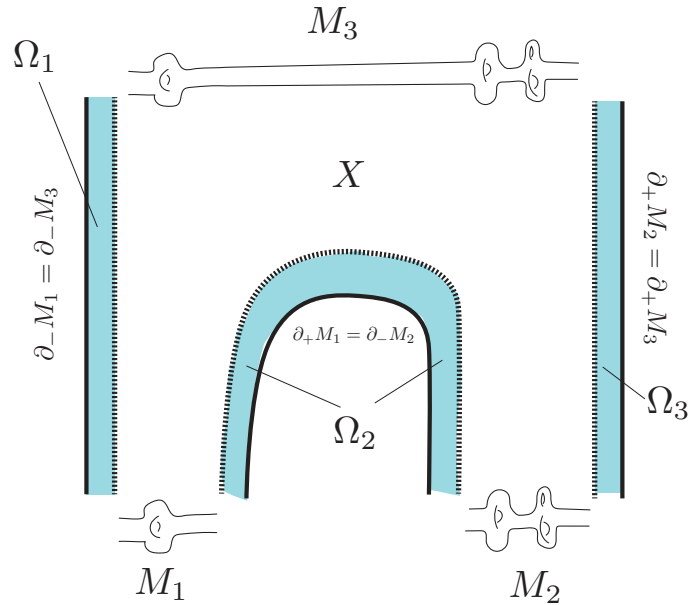}
\caption{Proof of Theorem \ref{tjhem83}}
\label{Figure62}
\end{figure}

\par\medskip
\noindent
{\bf Acknowledgement.} The research of the author is supported partially by NSF Grant No. 1406423
and Simons Collaboration for homological Mirror symmetry.
\bibliographystyle{amsalpha}

\end{document}